\documentclass[reqno]{amsart}
\usepackage[english]{babel}
\usepackage{latexsym}
\usepackage{amsfonts, amssymb}
\usepackage{graphicx}
\usepackage[a4paper,left=22mm,right=22mm]{geometry}
\usepackage{amsmath}
\usepackage{esint}
\usepackage[dvips]{epsfig}
\usepackage[latin1]{inputenc}
\usepackage{psfrag}
\usepackage{enumerate}
\usepackage[T1]{fontenc}

\newtheorem{theorem}{Theorem}[section]
\newtheorem{lemma}{Lemma}[section]
\newtheorem{proposition}{Proposition}[section]

\newtheorem{corollary}{Corollary}[section]

\newtheorem{definition}{Definition}[section]
\newenvironment{example}{\noindent\textbf{Example.~}}{}
\numberwithin{equation}{section}
\def\qed{\hfill$\diamond$\par \bigskip}

\usepackage[colorlinks=printing,backref]{hyperref}      % Print version

\renewcommand{\phi}{\varphi}

\newcommand{\Ll}[1]{\mathbf{L}^{#1}}
\newcommand{\caratt}[1]{{\displaystyle\chi^{\vphantom{\big\{}}_{\strut{\textstyle #1}}}}
\newcommand{\dirac}[1]{{\displaystyle\delta_{\strut{\textstyle #1}}}}
\newcommand{\Llloc}[1]{\mathbf{L^{#1}_{loc}}}
\newcommand{\C}[1]{\mathbf{C^{#1}}}
\newcommand{\Cc}[1]{\mathbf{C_c^{#1}}}
\newcommand{\modulo}[1]{{\left|#1\right|}}
\newcommand{\norma}[1]{{\left\|#1\right\|}}
\newcommand{\N}{{\mathbb{N}}}
\newcommand{\tv}{\mathrm{TV}}
\newcommand{\lip}{\mathrm{Lip}}
\newcommand{\BV}{\mathbf{BV}}
\newcommand{\reali}{\mathbb{R}}
\newcommand{\naturali}{\mathbb{N}}
\newcommand{\spt}{\mathop{\mathrm{spt}}}
\newcommand{\der}{\mathrm{d}}
\newcommand{\sgn}{\mathrm{sgn}}
\begin{document}
\allowdisplaybreaks

\title[Rigorous derivation of scalar conservation laws via many particle limit]{Rigorous derivation of nonlinear scalar conservation laws from follow-the-leader type models via many particle limit}


%\titlerunning{Rigorous derivation of scalar conservation laws via many particle limit}        % if too long for running head

\author{M.~Di Francesco}
\address{Marco Di Francesco - Department of Information Engineering, Computer Science, and Mathematics - University of L'Aquila, Via Vetoio 1, 67100 L'Aquila, Italy.}
\email{mdifrance@gmail.com}
\author{M.D.~Rosini}
\address{Massimiliano D. Rosini - ICM, University of Warsaw, ul.~Prosta 69, 00-838 Warsaw, Poland}
\email{mrosini@icm.edu.pl}

\begin{abstract}
We prove that the unique entropy solution to a scalar nonlinear conservation law with \emph{strictly monotone velocity} and nonnegative initial condition can be rigorously obtained as the large particle limit of a microscopic follow-the-leader type model, which is interpreted as the discrete Lagrangian approximation of the nonlinear scalar conservation law. More precisely, we prove that the empirical measure (respectively the discretised density) obtained from the follow-the-leader system converges in the $1$--Wasserstein topology (respectively in $\Llloc1$) to the unique Kruzkov entropy solution of the conservation law. The initial data are taken in $\Ll\infty$, nonnegative, and with compact support, hence we are able to handle densities with vacuum. Our result holds for a reasonably general class of velocity maps (including all the relevant examples in the applications, e.g.~in the Lighthill-Whitham-Richards model for traffic flow) with possible degenerate slope near the vacuum state. The proof of the result is based on discrete $\BV$ estimates and on a discrete version of the one-sided Oleinik-type condition. In particular, we prove that the regularizing effect $\Ll\infty \mapsto \BV$ for nonlinear scalar conservation laws is intrinsic of the discrete model.
\end{abstract}

\maketitle

\begin{center}
  \begin{tabular}{p{130mm}}
  \end{tabular}
\end{center}

\section{Introduction}\label{sec:Intro}

\subsection{Nonlinear scalar conservation laws in one space dimension}

%Conservation laws are applied to model the evolution of continuum quantities and arise in several areas of research.
A scalar conservation law in one space dimension for the unknown variable $\rho$ is a first order Partial Differential Equation (PDE) of the form
\begin{align}\label{eq:LWR_intro}
&\rho_t +f(\rho)_x = 0,&&t> 0, ~ x\in \reali,
\end{align}
see e.g.~the books~\cite{BressanBook,DafermosBook,SerreBook} as general references.
%The PDE~\eqref{eq:LWR_intro} is usually coupled with an initial condition $\rho(0,x)$ at time $t=0$.
%
%From the modelling point of view, opposite to a \emph{microscopic} (or atomistic) description, the unknown of a nonlinear conservation law describes a \emph{macroscopic quantity} such as, for instance, the local density or the local velocity of a fluid, the local temperature of an elastic body, or the local electrical charge density in a dielectric material. More specifically, the unknown $\rho$ in~\eqref{eq:LWR_intro} typically describes the evolution of the \emph{mass density} of a given medium with a one-dimensional structure.
The unknown $\rho$ typically describes the \emph{mass density} of a given medium with a one-dimensional structure.
%In this case, the equation~\eqref{eq:LWR_intro} itself is derived by assuming that the total (linear) mass of the medium in any interval $[a,b]$ and at any time $t>0$ can be expressed via a density function $\rho(t,x)$, $x \in [a,b]$.
In this case, a conservation law of the form~\eqref{eq:LWR_intro} can be derived under the hypothesis that the time evolution of the (linear) mass contained in an arbitrary interval $]a,b[ \subset \reali$ at time $t>0$, namely the integral $\int_a^b \rho(t,x) \, {\der} x$, is only affected by the \emph{flux} $f(\rho)$ of the medium at the edges of the interval, $x=a$ and $x=b$. In this sense, \eqref{eq:LWR_intro} can be seen as a \emph{continuity equation} for $\rho$, and the expression
\begin{equation}\label{eq:v_intro}
f(\rho)=\rho \, v(\rho)
\end{equation}
arises very naturally, where $\rho\mapsto v(\rho)$ is a constitutive law for the \emph{Eulerian velocity} of the medium.% The specific expression for $v$ depends on the application under study.

The formulation of the equation~\eqref{eq:LWR_intro} relies on the \emph{continuum assumption}, in which the medium is assumed to be indefinitely divisible without changing its physical nature.
Such assumption is needed to introduce the concept of macroscopic local density $\rho$, defined as the limit of the ratio $\Delta\ell/\Delta x$ as the measure of the elementary interval $\Delta x$ goes to zero, being $\Delta\ell$ the (linear) mass contained in $\Delta x$.
The continuum assumption is ensured by considering a \emph{very large mass} compared to the size of the domain occupied by the medium.

The density of vehicles in traffic flow is a typical example from real world applications in which an equation of the form~\eqref{eq:LWR_intro} is used as the `macroscopic counterpart' of a large system of moving vehicles. Traffic flow is indeed one of the main motivating applications behind the present paper, and we shall therefore describe the use of the continuum assumption in this context in Section~\ref{subsec:traffic_intro}.

The nonlinearity in the constitutive law $f=f(\rho)$ of the flux has made the mathematical theory for~\eqref{eq:LWR_intro} a challenging topic through the past decades. In Section~\ref{sec:LWRmodel} we recall the basic notions of the classical \emph{theory of entropy solutions} for scalar conservation laws, the foundations of which go back to the pioneering papers by Oleinik~\cite{oleinik} and Kruzkov~\cite{Kkruzkov}.

Nonlinear conservation laws have proven to feature many advantages in their interplay with applications (in particular in fluid dynamics and in traffic flow). First of all, they provide a very simple description of the \emph{shock structures} in the phenomena under study (see e.g.~compression waves in fluid mechanics, or queues in vehicular traffic). Most importantly, the mathematical theory developed in the past decades have greatly helped the development of efficient numerical schemes (and vice versa!), which are much handier and easier to manage when compared to agent based methods arising from a microscopic `moving particles' approach. Consequently, the continuum approach based on conservation laws allows to state and possibly solve optimal control and optimal management problems, and to easily extend the theory to more complex structures (such as networks).

\subsection{Goal of the paper: the many particle approximation of a scalar conservation law}

The goal of this paper is to prove that (under reasonable assumptions on the velocity map $v$, and assuming nonnegative initial data) a scalar nonlinear conservation law of the form~\eqref{eq:LWR_intro} can be \emph{solved} as a \emph{many particle limit} of a discrete (microscopic) model of interacting particle systems solving a suitable set of ODEs.
We emphasise here that the velocity map $\rho\mapsto v(\rho)$ will be assumed to be \emph{monotonically decreasing} w.r.t.~$\rho\geq0$, and with $v(0)=v_{\max}<+\infty$. The case of a monotonically increasing $v$ can be recovered easily by simple modifications in our construction.

Our approach can be sketched as follows. We fix in $\Ll1\cap \Ll\infty$ an initial density $\bar\rho \geq0$ with compact support and having total (linear) mass $L>0$. For a given integer $N>0$, we split the subgraph of $\bar\rho$ in $N$ adjacent regions of equal mass $\ell\doteq L/N$, with the endpoints of each region positioned at $\bar{x}_i \in \reali$, $i=0,\ldots,N$. The points $\bar{x}_i$ are interpreted as (ordered) particles with mass $\ell$, and they are taken as initial condition to an ODE system describing the evolution of the particles in the discrete setting, namely to the \emph{Follow-The-Leader (FTL) system}
\begin{subequations}\label{eq:ODEintro}
\begin{align}\label{eq:ODEintro1}
  & \dot{x}_{N}(t) = v_{\max},\\\label{eq:ODEintro2}
  & \dot{x}_{i}(t) = v\left(\frac{\ell}{x_{i+1}(t) - x_i(t)}\right),\qquad i=0,\ldots,N-1.
\end{align}
\end{subequations}
The points $x_i(t)$ are interpreted as moving particles on the real line. Basically, each particle moves with a velocity which is computed (through the map $v$) via the discrete density obtained by detecting the distance with the nearest right neighborhood of $x_i$ at time $t$. The last particle on the right $x_N$ has no other particles to its right, therefore the value of the discrete density on the right of $x_N$ is zero, and this justify the ODE~\eqref{eq:ODEintro1} for $x_N$.

We remark here that no collisions occur between the particles (hence overtaking is not allowed), as the distance between two consecutive points can be proven to satisfy an efficient lower bound by a fixed multiple of $\ell$ depending on the initial condition. We shall focus on this issue in Lemma~\ref{lem:1} below. We shall describe the FTL model in detail below in Section~\ref{sec:FTLmodel}.

After having solved \eqref{eq:ODEintro} for all times, we consider the empirical measure
\begin{equation}\label{eq:empirical_intro}
    \rho^N(t)=\ell\sum_{i=0}^{N-1} \delta_{x_i(t)},
\end{equation}
and prove in Theorem~\ref{thm:main} that its limit (in a measure sense to be explained later on) as $N$ goes to infinity is actually an $\Ll1$ density $\rho$, which satisfies the scalar conservation law~\eqref{eq:LWR_intro} with $f(\rho)=\rho \, v(\rho)$ in the Oleinik-Kru{\v{z}}kov entropy sense~\cite{Kkruzkov,oleinik_discontinuous}, see Definition~\ref{def:entropy_intro} below.

Our convergence result has a natural interpretation as a \emph{many particle limit} for the scalar conservation law~\eqref{eq:LWR_intro}. In this sense, it can be seen as an abstract particle method for~\eqref{eq:LWR_intro} which can be also applied in the context of numerics. On the other hand, the discrete model~\eqref{eq:ODEintro} can be also interpreted as a discrete \emph{Lagrangian} formulation of~\eqref{eq:LWR_intro}, which makes our result meaningful from a physical point of view, as it validates the use of the macroscopic model \eqref{eq:LWR_intro} in cases in which the microscopic dynamics are easier to justify compared to the macroscopic ones.

Although the literature on nonlinear conservation laws is extremely rich of effective numerical schemes (we mention here the pioneering work of Glimm~\cite{glimm} for systems, and the wave-front tracking algorithm proposed by Dafermos in~\cite{Daf72} and improved later on by Di~Perna~\cite{dip76} and Bressan~\cite{Bre92}, see~\cite{BressanBook} and the references therein for more details), to our knowledge the rigorous approximation of an entropy solution to a scalar conservation law by the empirical solution to an ODE system of Lagrangian particles in the spirit of~\eqref{eq:ODEintro} has not been covered yet. The recent paper~\cite{ColomboRossi} provides preliminary results, but it does not contain the needed estimates to justify the limiting procedure. The main novelty with respect to previous results in the literature is that our result is purely \emph{constructive}, in the sense that it can be considered as an alternative tool to actually \emph{solve} a scalar conservation law. No property of the limiting solution is used, except the uniqueness of entropy solutions in~\cite{Kkruzkov,oleinik_discontinuous} which is used to prove that the scheme has a unique limit. As a byproduct of our work, the Kruzkov entropy condition that allows to single out a unique solution to \eqref{eq:LWR_intro} can be now also intuitively justified by the fact that it is satisfied by our discrete particle approximation, and is inherited by the density $\rho$ obtained in the many particle limit.

Our approach differs from most of the numerical approaches to the solution to a scalar conservation law in that it interprets the microscopic limit as a\emph{ mean field limit} of a system of interacting particles with nearest neighbour type interaction, in the spirit of (locally and non-locally) interacting particles systems in statistical mechanics, probability, kinetic theory, mathematical biology, etc. In this sense, our result can be expressed in the framework of large (deterministic) particle limits with application to several contexts in fluid mechanics, see e.g. the classical references~\cite{dobrushin,morrey,onsager}. In one space dimension, a key result in the context of deterministic approximations is the one by Russo~\cite{russo}, which applies to the linear diffusion equation, in which the diffusion operator is replaced by a nearest neighbour interaction term (see also later generalizations to nonlinear diffusion in~\cite{maccamy}). A recent result which uses the same approach to nonlinear drift diffusion equations is presented in~\cite{matthes_osberger}. We also mention here the paper by Brenier and Grenier~\cite{grenier}, which provides a particle justification of the pressureless Euler system (and a particle approximation for a scalar conservation law, although with a completely different approach and interpretation). Our approach can be considered more in the spirit of~\cite{russo}, applied to a scalar conservation law of traffic type.

The existing numerical method for scalar conservation laws which most resembles our particle method is probably the wave-front tracking algorithm, in which the solution is approximated by a piecewise constant profile which is discontinuous on a finite number of moving fronts. Such a structure suggests the \emph{total variation} as the natural quantity to look at in order to perform efficient uniform estimates, and the space $\Ll1$ as the natural environment to set up the problem and to measure the error in the approximation procedure. In our case, the approximating sequence is a \emph{linear combination of Dirac's deltas}. Therefore, a \emph{measure topology} is needed to compare the approximating solution and its limit. We shall show that the most natural choice for such a topology is (a scaled version of) the \emph{$1$--Wasserstein distance}, see~\cite{AGS,villani}.

We emphasize that the particle approach presented here, as well as the methodology used, can be potentially adapted to detect more general macroscopic models by refining the microscopic formulation of the FTL type model. Moreover, simple boundary conditions can be achieved in the limit by simple modifications of the discrete model (e.g.~with entrances and exits).

\subsection{Formal derivation of the scheme: the use of Wasserstein distance}

The main advantage in using the topology induced by the Wasserstein distance (in our one-dimensional context) relies on its identification with the $\Ll1$--topology in the space of \emph{pseudo-inverses of cumulative distributions}. Roughly speaking, let $\rho$ be a nonnegative solution to~\eqref{eq:LWR_intro} with mass $L>0$, and let
\[F(t,x)\doteq\int_{-\infty}^x \rho(t,x) \,{\der}x\,\in [0,L],\]
be its primitive. The pseudo inverse variable
\begin{align*}
    &X(t,z)\doteq\inf\left\{x\in \reali \,\colon\, F(x)>z\right\},&&z \in \left[0,L\right[,
\end{align*}
formally satisfies the \emph{Lagrangian PDE}
\begin{equation*}
    X_t(t,z)=v\left(\frac{1}{X_z(t,z)}\right).
\end{equation*}
Now, if we replace the $z$--derivative of $X$ by the forward finite difference
\begin{equation*}
    X_z\approx \frac{X(t,z+\ell)-X(t,z)}{\ell} ,
\end{equation*}
and assume that $X$ is piecewise constant on intervals of length $\ell$, the ODE system~\eqref{eq:ODEintro} is immediately recovered, with the structure
\[X(t,z)=\sum_i x_i(t) \, \chi_{[i\ell,(i+1)\ell[}(z) .\]
We shall explain the above formal computation more in detail in Section~\ref{sec:formalderivation1} in the Appendix.

The use of pseudo-inverse variables and Wasserstein distances in the framework of scalar conservation laws is not totally new. In~\cite{CDL1}, a contraction estimate in the so-called $\infty$--Wasserstein distance for genuinely nonlinear scalar conservation laws was derived. The case of non-decreasing solution was treated earlier in~\cite{BBL}. In the special case of the LWR equation for traffic flow, we also remark here that in~\cite{Newell} a simplified version of the model~\eqref{eq:LWR_intro} is derived by introducing as new variable the cumulative number of vehicles passing through a location $x$ at time $t$ starting from the passage of some reference vehicle, see~\cite{Aubin2010963,Daganzo2005187} for recent developments of this theory.

\subsection{Technical aspects of the problem. The discrete Oleinik condition}\label{subsec:intro_technical}

From the technical point of view, our convergence result relies first of all on proving that the empirical measure~\eqref{eq:empirical_intro} has the same (weak) $N\rightarrow +\infty$ limit as the piecewise constant approximation
\begin{equation*}
    \hat{\rho}^N(t,x)=\sum_{i+1}^{N-1}y_i(t)\chi_{[x_i(t),x_{i+1}(t)[},\qquad y_i(t)\doteq \frac{\ell}{x_{i+1}(t)-x_i(t)},
\end{equation*}
in which $y_i(t)$ is the discrete Lagrangian version of the density. The most important step, however, lies in providing strong $\Ll1$ compactness of $\hat{\rho}^N$. This task is performed in two different ways. In the case of $\BV$ initial data, we are able to provide a direct estimate of the total variation of the discrete density (see Proposition~\ref{pro:compactness1}). On the other hand, our main result concerns with the case of general $\Ll\infty$ data: in this case, a key estimate on the particle model (see Lemma~\ref{lem:oleinik}), which can be considered as a discrete version of the Oleinik condition for the scalar conservation law, allows to provide strong compactness even if the initial total variation is unbounded. In some sense, this proves that the one-sided Lipschitz regularizing effect of the scalar conservation law~\eqref{eq:LWR_intro} is somehow an intrinsic property of the discrete Lagrangian formulation of the model. We defer to~\cite{GolsePerthame2013} and the references therein for general results on the regularizing effect for scalar conservation laws.

For numerical purposes, the use of discrete Oleinik conditions has been addressed before for the Lax-Friedrichs and Godunov schemes in~\cite{osher,leveque,oleinik_discontinuous,tadmor}. There is also a similar result for second order systems in~\cite{degond_rascle}. The striking novelty in our approach is the fact that our discrete Oleinik condition is only posed in terms of the velocity field, whereas the classical Oleinik condition is stated in terms of the derivative of the flux, see~\cite{hoff83}. This is due to the fact that the discrete model is a Lagrangian one, and is therefore characterised by the velocity law. The advantage of having the discrete one-sided Lipschitz condition in terms of the velocity is that we can also consider velocity laws with degenerate slopes at $\rho=0$. An interesting numerical feature (which is however quite natural when considering particle approximations) is that the discrete approximation $\hat{\rho}^N$ for the density has \emph{no vacuum regions} in the interior of its support, no matter whether or not the (continuum) initial condition is made up by more than one hump. Finally, let us mention that our discrete density $\hat{\rho}^N$ is always discontinuous on at most $N+1$ fronts, unlike in the wave front tracking approximation in which the number of jumps may increase in time.

\subsection{A motivating example: traffic flow}\label{subsec:traffic_intro}

The macroscopic variables describing vehicular traffic are the (mean) density $\rho$ (number of vehicles per unit length of the road), the (mean) velocity $v$ (space covered per unit time by the vehicles) and the flux $f$ (number of vehicles per unit time).
By definition we have that
\[f = \rho \, v .\]
The conservation of the number of vehicles along a road with no entrance or exit is expressed by the PDE
\[\rho_t +f_x = 0,\]
where $t$ is the time and $x$ the position along the road.
To close the above system of two equations and three unknowns, a further condition has to be imposed.
Two main approaches are used in the literature (see~\cite{Hoogendoorn01state-of-the-artof,PiccoliTosinsurvey,Rosinibook} for a survey): first order and second order models.
The former are based on a constitutive law, which expresses one of the three unknowns as function of the remaining two.
In the latter, a further evolution equation is imposed.
The prototype of first order models is the Lighthill-Whitham-Richards model (LWR)~\cite{LighthillWhitham,Richards},
which assumes that the velocity can be expressed as an explicit function of the density alone. Hence, LWR is nothing but~\eqref{eq:LWR_intro}-\eqref{eq:v_intro} with  the choice of $v$ depending on the specific situation, see Example~\ref{rem:example_velocities} below. The most celebrated second order model is the Aw-Rascle-Zhang model (ARZ)~\cite{Aw2000916,Zhang2002275}, which adds an evolution equation that can be regarded as a continuum analogue of Newton's law.

The continuum assumption is not immediately justifiable in the context of vehicular traffic, as the number of vehicles is typically far lower than the typical number of molecules e.g.~in fluid dynamics. Usually, the continuum hypothesis is accepted as a technical approximation of the physical reality, regarding macroscopic quantities as measures of traffic features.
In order to justify and make more clear the continuum hypothesis, the study of the discrete-to-continuum limit for second order models has been proposed in~\cite{AKMR2002,degond_rascle}. First attempts at analyzing the particle approximation for first order macroscopic models have been recently proposed in~\cite{ColomboMarson,ColomboRossi,Rossi}. From the modelling point of view, the discrete-to-continuum limit in the context of traffic flow may be also considered as the theoretical counterpart of reconstructing the traffic state of a region through high-sampling data from GPS devises. Continuum traffic models can be also detected from the microscopic scale delivered by kinetic models, see e.g.~the approach proposed in~\cite{degond_kinetic} leading to a variant of the model of~\cite{Aw2000916}. We mention also the recent~\cite{bellomo_bellouquid}, in which interactions at the microscopic scale are modeled by methods of game theory, thus leading to the derivation of mathematical models within the framework of the kinetic theory.

The FTL approximation scheme has been also used to approximate second order models for vehicular traffic such as ARZ~\cite{Aw2000916,Zhang2002275}. An important result in this sense is the one in~\cite{AKMR2002}. However, differently from~\cite{AKMR2002}, in our result we do not shrink the length of the vehicles to zero and we do not let the size of the highway or the number of vehicles under consideration tend to infinity. In fact, our approximation algorithm rather lets the number of particles (platoons of possible fractional vehicles) under consideration tend to infinity, but keeps both the length of the domain (highway) and the total mass (total number of vehicles) $L$ constant. Finally, another important difference from~\cite{AKMR2002} is that our approach allows to handle \emph{vacuum regions}. This introduces further technical difficulties that are rigourously treated and solved in the present paper.

%The approximation of first order models in vehicular traffic via discrete models in a rigorous and constructive form is one of the motivating goal of our work. However, vehicular traffic is only one of the contexts in which our result can be stated and applied.

\subsection{Structure of the paper}

The present paper is structured as follows. Section~\ref{sec:allpreliminaries} contains all the preliminary material and the statement of the main result. More precisely, we summarise the mathematical theory of the one-dimensional scalar conservation law~\eqref{eq:LWR_intro} in Section~\ref{sec:LWRmodel}, we define the (discrete) FTL~model~\eqref{eq:ODEintro} in detail in Section~\ref{sec:FTLmodel}. we recall the basics on the Wasserstein distance in one space dimension in Section~\ref{sec:preliminaries}, and we set up the approximating scheme and state our main result in Section~\ref{subsec:statement}. The precise statement of the main result is contained in Theorem~\ref{thm:main}. Section~\ref{sec:result} is devoted to its proof, and is split into the subsections~\ref{sec:convergence}, \ref{sec:Oleinik}, \ref{sec:compactness}, and~\ref{sec:convergence2}. More precisely, Section~\ref{sec:convergence} is devoted to the proof of the weak convergence of our approximating scheme, in Section~\ref{sec:Oleinik} we prove the two basic compactness estimates mentioned above, in Section~\ref{sec:compactness} we provide the needed time-continuity and prove strong compactness in $\Ll1$, and finally in Section~\ref{sec:convergence2} we prove that the limit of our approximating scheme is the unique entropy solution in the Oleinik-Kru{\v{z}}kov entropy sense~\cite{Kkruzkov,oleinik_discontinuous}.

\section{Preliminaries and statement of the main result}\label{sec:allpreliminaries}

\subsection{One-dimensional scalar conservation law}\label{sec:LWRmodel}

Consider the Cauchy problem for a scalar conservation law in one space dimension
\begin{subequations}\label{eq:LWRmodel}
\begin{align}\label{eq:LWR_intro1}
    &\rho_t + \left[\rho \,v(\rho)\right]_x=0,&&t>0,~ x\in\reali,
    \\ \label{eq:LWR_intro2}
    &\rho(0,x)= \bar\rho(x),&&x\in \reali,
\end{align}
\end{subequations}
where $\rho = \rho(t,x) \geq 0$ represents the density of the medium at the position $x \in\reali$ at time $t\ge0$, $v=v(\rho)$ is the local velocity and $\bar\rho$ is the initial datum, see e.g.~\cite{BressanBook,DafermosBook,SerreBook} as general references.
Regularity (and convexity) assumptions are typically required in the literature on the \emph{flux} function $f(\rho) \doteq \rho \,v(\rho)$.
However, \emph{our assumptions will be on the velocity $v$ rather than on the flux $f$}, see Section~\ref{subsec:statement}.

%A typical requirement in the classical theory of scalar conservation laws is that the \emph{flux} function $$f(\rho) \doteq \rho \,v(\rho)$$ is assumed to be concave (or convex) on the whole domain for the state variable $\rho$. We shall state our assumptions on $f$ in Section~\ref{subsec:statement}. More precisely, since we shall construct our solution $\rho$ via the discrete solution to~\eqref{eq:FTL_intro}, \emph{our assumptions will be on the velocity map $v$ rather than on the flux $f$}.

Let us briefly recall the main properties of the solutions to~\eqref{eq:LWRmodel}.
If the initial datum has compact support, then the support of any solution has finite speed of propagation. The maximum principle holds true, and if the initial datum is nonnegative, then the solutions remain nonnegative for all times. Moreover, the total (linear) mass $\int_\reali \rho(t,x) \, {\der} x$ is time independent: $\int_\reali  \rho(t,x) \, {\der} x = L \doteq \int_\reali  \bar \rho(x) \,  {\der} x$ for all $t\ge0$.

It is well known that the solutions to~\eqref{eq:LWRmodel} may develop discontinuities in a finite time, also for regular initial data. For this reason, one has to consider weak solutions $\rho$ to~\eqref{eq:LWRmodel}, more precisely $\rho$ in $\Ll\infty\left(\left[0,+\infty\right[; \Ll1\cap \Ll\infty\left(\reali\right)\right)$ that satisfy~\eqref{eq:LWRmodel} in the sense of distributions, namely
\begin{align}\label{eq:weaksol}
    &\int_{\reali}\int_{\reali_+} \left[\vphantom{\int}\rho(t,x) \,\varphi_t(t,x) + f\left(\rho(t,x)\right) \varphi_x(t,x)\right] {\der}t \,{\der}x +\int_{\reali}\bar\rho(x)\, \varphi(0,x)\,{\der}x=0
\end{align}
for all $\varphi \in \Cc\infty\left(\left[0,+\infty\right[ \times \reali; \reali\right)$. The choice of $\Ll1\cap \Ll\infty\left(\reali\right)$ as the functional space to deal with the $x$--regularity appears as the most reasonable one in order to obtain existence of weak solutions when the approximating procedure is performed via a vanishing viscosity argument, see e.g.~\cite[Section~6.3]{DafermosBook}. However, the space $\BV\left(\reali\right)$ is more reminiscent of the typical structure of solutions featuring shocks and rarefaction waves, and turns out to be a natural choice when the problem is e.g.~solved by the polygonal approximation algorithm also known as the wave-front tracking algorithm~\cite{Daf72}, see also~\cite{BressanBook} and the references therein.

It is well known that the notion of weak solution introduced above is not strong enough to provide uniqueness of solutions to~\eqref{eq:LWRmodel}. The concept of entropy solution formulated in~\cite{Kkruzkov,lax,oleinik_discontinuous} (see also~\cite{DafermosBook} and the references therein), provides the most natural and efficient way to single out a unique (physically relevant) solution to~\eqref{eq:LWRmodel}. Such concept can be be formulated in several ways, also depending on the regularity of $\rho$, the most general one being the one proposed by Kru{\v{z}}kov~\cite{Kkruzkov}, which holds for a reasonably wide class of fluxes (namely $\rho\mapsto f(\rho)$ being locally Lipschitz) and in arbitrary space dimension.

\begin{definition}[Entropy solutions]\label{def:entropy_intro}
    Assume that the flux $\rho\mapsto f(\rho)$ is locally Lipschitz. A function $\rho$ in $\Ll\infty\left(\left[0,+\infty\right[; \Ll1\cap \Ll\infty\left(\reali\right)\right)$ is an \emph{entropy solution} to~\eqref{eq:LWRmodel} if it satisfies the entropy inequality
    \begin{align}
      & \int_{\reali}\int_{\reali_+} \left[\vphantom{\int}\modulo{\rho(t,x) - k} \,\varphi_t(t,x) + \sgn\left(\rho(t,x) - k\right) \left[f\left(\rho(t,x)\right) - f(k)\right] \varphi_x(t,x)\right] {\der}t \,{\der}x \nonumber\\
      & \ + \int_{\reali} \modulo{\bar\rho(x) -k} \, \varphi(0,x) \,{\der}x \ge 0\label{eq:entropy_ineq}
    \end{align}
    for all $\varphi \in \Cc\infty\left(\left[0,+\infty\right[ \times \reali; \reali\right)$ with $\varphi\ge 0$, and for all constants $k\in \reali$.
\end{definition}\noindent
Clearly, any entropy solution is a weak solution to~\eqref{eq:LWRmodel} in the sense of~\eqref{eq:weaksol}. Moreover uniqueness follows from~\eqref{eq:entropy_ineq}.
\begin{theorem}[Kru{\v{z}}kov~\cite{Kkruzkov}]\label{thm:kruzkov}
Assume that the flux $f$ is locally Lipschitz. Then, for any given initial condition $\bar\rho$ in $\Ll\infty$ with compact support, there exists a unique entropy solution to~\eqref{eq:LWRmodel} in the sense of Definition~\ref{def:entropy_intro}.
\end{theorem}
It is easy to check that any function $\rho$ satisfying the entropy inequality~\eqref{eq:entropy_ineq} satisfies also
the following property of (weak) $\Ll1$--continuity in time
\begin{equation*}
    \lim_{T\to0+} \frac{1}{T} \int_0^T \int_{\modulo{x}\le r} \modulo{\rho(t,x)-\bar\rho(x)} \,{\der}x \,{\der}t = 0
\end{equation*}
for all $r>0$. However, depending on the way we attempt at constructing entropy solutions, an important issue is related with detecting the trace at $t=0$ in a strong enough topology. This is often the case when the approximating scheme lacks of compactness when $t$ approaches zero. A theorem due to Chen and Rascle~\cite{chen_rascle} states that the uniqueness of the entropy solution is preserved also for a notion of entropy solution relaxed at $t=0$, provided the flux $f$ satisfies a.e.~a genuine nonlinearity condition.

\begin{theorem}[Chen and Rascle~\cite{chen_rascle}]\label{thm:chen}
Assume there exists no nontrivial interval on which $f$ is affine. If $\bar\rho$ is in $\Ll\infty$ and has compact support, then there exists a unique $\rho$ in $\Ll\infty\left(\left[0,+\infty\right[; \Ll1\cap \Ll\infty\left(\reali\right)\right)$ weak solution to~\eqref{eq:LWRmodel} in the sense of~\eqref{eq:weaksol} that satisfies also
\begin{align}
      & \int_{\reali}\int_{\reali_+} \!\!\left[\vphantom{\int}\modulo{\rho(t,x) - k} \,\varphi_t(t,x) + \sgn(\rho(t,x) - k)[f\left(\rho(t,x)\right) - f(k)] \,\varphi_x(t,x)\right]\! {\der}t \,{\der}x \!\ge 0\label{eq:entropy_ineq_notrace}
    \end{align}
for all $\varphi \in \Cc\infty\left(\left]0,+\infty\right[ \times \reali; \reali\right)$ with $\varphi\ge 0$, and for all constants $k\in \reali$. Moreover, $\rho$ is the unique entropy solution in the sense of Definition~\ref{def:entropy_intro}.
\end{theorem}

Let us finally recall that, for $\C1$--fluxes $f$ which are concave or convex, another classical tool to uniquely determine all weak solutions by their $\Ll\infty$--initial values is the so called Oleinik-type condition~\cite{hoff83}
\begin{align}\label{eq:ole_intro_hoff}
    & \int_{\reali}\int_{\reali_+}f'(\rho(t,x)) \,\varphi_x(t,x)\,{\der}t \,{\der}x
    \ge - \int_{\reali}\int_{\reali_+} \frac{1}{t} \, \varphi(t,x)\,{\der}t \,{\der}x
\end{align}
for all $\varphi \in \Cc\infty\left(\left[0,+\infty\right[ \times \reali; \reali\right)$ with $\varphi\ge 0$, and for all $t>0$.
Moreover, if $f'$ has Lipschitz continuous inverse, then~\eqref{eq:ole_intro_hoff} implies that $\rho(t,\cdot)$ has locally bounded total variation for all $t>0$ even if the initial datum is not locally in $\BV$.

\subsection{Follow-the-leader model}\label{sec:FTLmodel}

Microscopic models of vehicular traffic are typically based on the so called Follow-The-Leader (FTL) model. On the other hand, such model makes sense in any context in which a set of discrete ordered agents moves with a velocity computed instantaneously as a function of the discrete density.

Consider $N+1$ ordered particles localised on $\reali$. Denote by $t \mapsto x_i(t)$ the position of the $i$--th particle for $i = 0, \ldots, N$. Then, according to the FTL~model, the evolution of the particles (which mimics the evolution of the position of $N+1$ vehicles along the road) is described inductively by the following Cauchy problem for an ODE system
\begin{subequations}\label{eq:FTL_intro}
\begin{align}\label{eq:FTL_intro1}
  &\dot{x}_N(t) = v_{\max} ,\\
  \label{eq:FTL_intro2}
  &\dot{x}_i(t) = v\left(\frac{\ell}{x_{i+1}(t)-x_i(t)}\right),&
  &i=0,\ldots,N-1,\\
  \label{eq:FTL_intro3}
  &x_i(0) = \bar{x}_i ,&
  &i=0,\ldots,N,
\end{align}
\end{subequations}
where $v$ is a $\C1$ strictly decreasing velocity map, $\bar{x}_0 <  \ldots < \bar{x}_N$ are the initial positions of the particles, and $\ell>0$ is the (linear) mass of each particle. We shall assume that $v$ is strictly decreasing and bounded from above. Then, $v_{\max} \doteq v(0)$ is the maximum possible velocity and is reached only by the leading particle $x_N$. Let us underline that in vehicular traffic the FTL model requires that all vehicles move towards a unique direction (i.e.~the vehicles move along a one way road), e.g.~$v$ is nonnegative. Our setting is more general and such an assumption is not required. Notice in particular that we shall not prescribe any constrain on the sign of $v_{\max}$.

We now set
\begin{equation*}
  R \doteq \max_{i=0,\ldots,N-1}\left(\frac{\ell}{\bar{x}_{i+1} - \bar{x}_{i}}\right).
\end{equation*}
The quantity $R>0$ is the \emph{maximum discrete density} at time $t=0$ and
\begin{align}\label{eq:FTL_assumption1}
    &\bar{x}_{i+1} - \bar{x}_{i} \ge \frac{\ell}{R},& i = 0, \ldots, N-1.
\end{align}

System~\eqref{eq:FTL_intro} can be solved inductively starting from $i=N$. Indeed, from~\eqref{eq:FTL_intro1}, we immediately deduce that
\begin{align*}
    x_{N}(t) = \bar{x}_{N} + v_{\max} \,t .
\end{align*}
Then, we can compute $t \mapsto x_i(t)$ once we know $t \mapsto x_{i+1}(t)$. In fact, according with the system~\eqref{eq:FTL_intro} the velocity of the $i$--th particle depends on its distance from the $(i+1)$--th particle alone via the smooth velocity map $v$. In order to ensure that the (unique) solution to~\eqref{eq:FTL_intro} exists globally in $t\geq 0$, we need to prove that the distances $x_{i+1}(t)-x_i(t)$ never degenerate. We can actually prove that the discrete density never exceeds the upper bound $R>0$ holding at $t=0$. This is proven in the next lemma, which improves a similar one in~\cite{Rossi}, which only holds in the case $v(\rho_{\max})=0$ for some maximal density $\rho_{\max}>0$, and in which the authors prove that the discrete density stays bounded by the threshold density $\rho_{\max}$. Our result is actually much stronger, because it basically proves that the discrete model satisfies the same maximum principle of the corresponding continuum model (see e.g.~\cite[Theorem 6.2.4]{DafermosBook}), i.e.~the discrete density at an arbitrary time $t>0$ is controlled by the supremum norm of the initial discrete density.

\begin{lemma}[Discrete maximum principle]\label{lem:1}
    For all $i = 0, \ldots, N-1$, we have
    \begin{align}\label{est:lem1}
        &\frac{\ell}{R} \le x_{i+1}(t)-x_i(t) \le \bar{x}_{N} - \bar{x}_{0} + (v_{\max}-v(R))\,t &\hbox{for all times }t\ge0.
    \end{align}
\end{lemma}
\begin{proof}
The upper bound is obvious. Hence, it suffices to prove the lower bound.
At time $t=0$ the lower bound is satisfied because of~\eqref{eq:FTL_assumption1}. We shall prove that
\begin{align}\label{eq:maximum_discrete_statement}
&\sup_{t\geq 0} \left [x_{j+1}(t)-x_j(t)\right ] \geq \frac{\ell}{R},&
&j=0,\ldots,N-1,
\end{align}
by a recursive argument on $j$. The statement is true for $j=N-1$. Indeed,
\begin{align*}
   x_N(t)-x_{N-1}(t)& = \bar{x}_N-\bar{x}_{N-1} + \int_0^t \left[v_{\max}-v\left(\frac{\ell}{x_N(s)-x_{N-1}(s)}\right)\right] {\der} s\\
  & \geq \bar{x}_N-\bar{x}_{N-1}\geq \frac{\ell}{R},
\end{align*}
because $v(\rho)\leq v_{\max}$ for all $\rho\geq 0$. Assume now that
\begin{equation}\label{eq:maximum_principle_induction}
   \sup_{t\geq 0} \left[x_{j+2}(t)-x_{j+1}(t)\right] \geq \frac{\ell}{R}.
\end{equation}
Assume by contradiction that there exists $j \in \left\{0,\ldots,N-2\right\}$ and $t_2>t_1\ge0$ such that
\[x_{j+1}(t_1)-x_j(t_1) = \frac{\ell}{R}\]
and
\begin{align}\label{eq:contradiction_maximum}
&x_{j+1}(t)-x_j(t) < \frac{\ell}{R}&
&\hbox{for all }t \in \left]t_1, t_2\right].
\end{align}
Since $v$ is strictly decreasing, we have for any $t \in \left]t_1, t_2\right]$
\begin{equation}\label{eq:basic}
x_j(t)
=
x_j(t_1) + \int_{t_1}^t v\left(\frac{\ell}{x_{j+1}(t)-x_j(t)}\right) {\der} t
\leq
x_j(t_1) + v(R) \, (t-t_1).
\end{equation}
By~\eqref{eq:maximum_principle_induction} we have for any $t \in \left]t_1, t_2\right]$
\[
x_{j+1}(t)
=
x_{j+1}(t_1) + \int_{t_1}^t v\left(\frac{\ell}{x_{j+2}(t)-x_{j+1}(t)}\right) {\der} t
\geq
x_{j+1}(t_1) + v(R) \, (t-t_1)\,,
\]
and therefore by~\eqref{eq:basic}
\[
x_{j+1}(t) - x_{j}(t)
\geq
x_{j+1}(t_1) - x_{j}(t_1) = \frac{\ell}{R},
\]
which contradicts~\eqref{eq:contradiction_maximum}. Hence, \eqref{eq:maximum_discrete_statement} is satisfied and the assertion is proven.
\end{proof}
We emphasise that the above discrete maximum principle is a direct consequence of the \emph{transport} nature behind the FTL system \eqref{eq:FTL_intro}. Indeed, the global bound for the discrete density is propagated from the last particle $x_N$ back to all the other particles, as emphasised by the recursive argument in the proof of Lemma \ref{lem:1}.

\subsection{Notation and preliminaries on measure distances}\label{sec:preliminaries}

In this section we recall basic properties of pseudo-inverse operators and on the one-dimensional Wasserstein distance that we shall use extensively in the rest of the paper. We defer to~\cite{villani} for further details.

For a fixed $L>0$, introduce the pseudo-inverse operators
\begin{align*}
    & \mathcal{X}\colon \Ll\infty\left(\reali;\left[0,L\right]\right) \rightarrow \Ll\infty\left(\left[0,L\right[;\reali\right),\\
    & \mathcal{F}\colon \Ll\infty\left(\left[0,L\right[;\reali\right) \rightarrow \Ll\infty\left(\reali;\left[0,L\right]\right),
\end{align*}
defined by
\begin{align*}
    &\mathcal{X}\left[F\right](z)  \doteq  \inf\left\{ x \in \reali \,\colon\, F(x)>z \right\}&\hbox{ for }&
    z \in \left[0,L\right[,
    \\
    &\mathcal{F}\left[X\right](x)  \doteq  \textrm{meas}\left\{ z \in \left[0,L\right] \colon\, X(z)\le x \right\}&\hbox{ for }&
    x\in\reali,
\end{align*}
and consider the space
\begin{equation}\label{eq:ML}
  \mathcal{M}_L\doteq\left\{\rho \hbox{ Radon measure on $\reali$ with compact support} \,\colon\, \rho\ge 0,\;\rho(\reali)=L\right\}.
\end{equation}
For a given $\rho\in \mathcal{M}_L$, we denote $x_{\min}^\rho \doteq \min\left(\spt (\rho)\right)$ and $x_{\max}^\rho \doteq \max\left(\spt (\rho)\right)$, and by $F_\rho \,\colon\, \reali \to [0,L]$ its cumulative distribution, namely
\[F_\rho(x) \doteq \rho\left(\left]-\infty,x\right]\right).\]
We observe that $F_\rho \in \Ll\infty\left(\reali;\left[0,L\right]\right)$ is non-decreasing, right-continuous with $F_\rho(x)=0$ for all $x< x_{\min}^\rho$ and $F_\rho(x)=L$ for all $x\ge x_{\max}^\rho$. Therefore we can define its pseudo-inverse $X_\rho \doteq \mathcal{X}\left[F_\rho\right]$. Clearly, $X_\rho \in \Ll\infty\left(\left[0,L\right[;\left[ x_{\min}^\rho, x_{\max}^\rho\right]\right)$ is non-decreasing, right-continuous with $X_\rho(0)=x_{\min}^\rho$. By abuse of notation, we shall adopt the notation $\rho$ to denote an absolutely continuous measure in $\mathcal{M}_L$ with $\Ll1$--density $\rho$. We recall the following lemma (see e.g.~\cite{villani}).

\begin{lemma}[Change of variable]\label{lem:change_of_variable}
If $\rho \in \mathcal{M}_L$, then for all $\varphi\in \C0(\reali;\reali)$ we have
\begin{equation*}
  \int_\reali  \varphi(x) \,{\der}\rho(x)= \int_0^L \varphi\left(X_\rho(z)\right)  {\der}z.
\end{equation*}
\end{lemma}

We recall that, for $L=1$, the one-dimensional \emph{$1$--Wasserstein distance} between $\rho_1,\rho_2\in \mathcal{M}_1$ (defined in terms of optimal plans in the Monge-Kantorovich problem, see e.g.~\cite{villani}) can be defined as
\begin{equation*}
    d_1(\rho_1,\rho_2)\doteq\norma{F_{\rho_1}-F_{\rho_2}}_{\Ll1(\reali;\reali)} = \norma{X_{\rho_1}-X_{\rho_2}}_{\Ll1([0,1];\reali)}.
\end{equation*}
For a general strictly positive $L$, we introduce the \emph{scaled $1$--Wasserstein distance} between $\rho_1,\rho_2\in \mathcal{M}_L$ as
\begin{equation}\label{eq:wass_equiv0}
    d_{L,1}(\rho_1,\rho_2)\doteq\norma{F_{\rho_1}-F_{\rho_2}}_{\Ll1(\reali;\reali)} = \norma{X_{\rho_1}-X_{\rho_2}}_{\Ll1([0,L];\reali)} .
\end{equation}
Indeed, straightforward computation yields
\begin{equation*}
    d_{L,1}(\rho_1,\rho_2) = L \, d_1(\rho_1/L,\rho_2/L).
\end{equation*}
The distance $d_{L,1}$ inherits all the topological properties of the $1$--Wasserstein distance for probability measures. In particular, a sequence $(\rho_n)_{n\in\N}$ in $\mathcal{M}_L$ converges to $\rho\in \mathcal{M}_L$ in $d_{L,1}$ if and only if
\begin{equation*}
    \lim_{n\to+\infty}\int_\reali \varphi(x) \,{\der}\rho_n(x) = \int_\reali \varphi(x) \,{\der}\rho(x) ,
\end{equation*}
for all $\varphi\in \C0(\reali;\reali)$ growing at most linearly at infinity.

\subsection{Statement of the main result}\label{subsec:statement}

In this subsection we state our main result, which provides a rigorous description of the unique entropy solution $\rho$ to the Cauchy problem~\eqref{eq:LWRmodel} as the limit for $N$ that goes to infinity of a density associated to the microscopic model~\eqref{eq:FTL_intro} to be constructed as described below.

We shall work under the standing assumption on the initial datum
\begin{enumerate}[(InBV)]
  \item[(In)] The initial datum $\bar\rho$ is in $\mathcal{M}_L \cap \Ll\infty(\reali)$,
\end{enumerate}
where $\mathcal{M}_L$ is defined in \eqref{eq:ML}. In some cases we shall require the stronger condition
\begin{enumerate}[(InBV)]
  \item[(InBV)] The initial datum $\bar\rho$ is in $\mathcal{M}_L \cap \BV(\reali)$.
\end{enumerate}

As for the velocity function $v$, we shall require throughout the paper:
\begin{enumerate}[(InBV)]
        \item[(V1)] $v\in \C1([0,+\infty[)$, $v$ strictly decreasing on $[0,+\infty[$.
        \item[(V2)] $v(0)=v_{\max}$ for some $v_{\max}\in \reali$.
\end{enumerate}
The $\C1$ assumption in (V1) is a minimal requirement for having a unique local solution to the system~\eqref{eq:FTL_intro}. The monotonicity assumption in (V1) is a natural one in traffic models. However, we shall see that it is technically relevant in many parts of our paper, in particular in Section~\ref{sec:Oleinik}. The assumption (V2) basically states that $v(0)$ is finite. We notice that the maximum principle for \eqref{eq:LWRmodel} and the discrete maximum principle in Lemma \ref{lem:1} imply, together with the monotonicity of $v$, that both $v(\rho)$ and $v(\ell/(x_{i+1}(t)-x_i(t)))$ are globally bounded in time by a constant depending on the $\Ll\infty$ norm of the initial datum. The assumption $v(0)<+\infty$ is a natural requirement in order to guarantee the existence of a solution to the FTL system~\eqref{eq:FTL_intro}. Indeed, if $v(0)=+\infty$, then the particle $x_N$ goes to $+\infty$ instantaneously after $t=0$, and so do all other particles. Therefore, the particle system cannot approximate the continuum equation \eqref{eq:FTL_intro} in the many particle limit, and the assumption $v(0)<+\infty$ is a necessary condition for our result.
Let us also remark that the assumption (V1) implies that $f(\rho)=\rho \, v(\rho)$ is not affine, and hence satisfies the hypotheses of Theorem \ref{thm:chen}.% Indeed, since $f(0)=0$, assuming by contradiction that $f$ is affine we get $f(\rho)=\rho \, v(\rho)= A\rho$, i.e.~$v$ is constant, which contradicts with (V1).

In some cases, we shall use the stronger assumption
\begin{enumerate}[(InBV)]
 \item[(V3)] The map $[0,+\infty[\,\, \ni\rho\,\mapsto \, \rho\,v'(\rho) \in [0,+\infty[$ is non increasing.
\end{enumerate}
Notice that the assumption~(V3) implies in particular that the flux $\rho\mapsto f(\rho)=\rho \, v(\rho)$ is concave. Indeed, since $f'(\rho)=v(\rho)+ \rho \, v'(\rho)$, $f'$ is the sum of two non-increasing functions if (V3) is satisfied. The assumption (V3) is clearly a stricter requirement than the concavity of $f$; on the other hand, (V3) is verified by many examples of velocities arising in traffic flow models.

\begin{example}[Examples of velocities in vehicular traffic]\label{rem:example_velocities}
In vehicular traffic, a maximal density $\rho_{\max}>0$ is prescribed, at which all vehicles are bumper to bumper. Typically, $\rho$ in \eqref{eq:LWRmodel} represents a normalised density, and one can assume $\rho_{\max}=1$.
The prototype for the velocity in vehicular traffic $v(\rho) = v_{\max} \left(1-\rho\right)$ by Greenshields~\cite{Greenshields} clearly satisfies the assumptions~(V1), (V2), (V3). The same holds for the Pipes-Munjal velocity~\cite{pip}
\begin{align*}
    &v(\rho) = v_{\max} \, \left(1 - \rho^\alpha\right) \qquad \alpha>0,
\end{align*}
in which the concavity of the flux $\rho \, v(\rho)$ degenerates at $\rho=0$, and for the Underwood model~\cite{Underwood}
\begin{align*}
    &v(\rho) = v_{\max}\, e^{-\rho} .
\end{align*} A further example of speed-density relations that satisfy~(V1), (V2), (V3) are
\begin{align*}
    &v(\rho) = v_{\max} \, \left[\log\left(\frac{1}{\alpha}\right)\right]^{-1} \log\left(\frac{1}{\rho+\alpha}\right) ,&
    &\alpha >0,
\end{align*}
that result from a slight modification of the Greenberg model~\cite{Greenberg}
\end{example}

We now introduce our atomization scheme.
Denote by $\bar{x}_{\min} < \bar{x}_{\max}$ the extremal points of the convex hull of the support of $\bar\rho$, namely $\bigcap_{\left[a,b\right]\supseteq\spt\left(\bar\rho\right)} \left[a,b\right] = \left[\bar{x}_{\min}, \bar{x}_{\max}\right]$.
Fix $n \in \naturali$ sufficiently large. Let $L \doteq \norma{\bar\rho}_{\Ll1(\reali)} > 0$ and split the subgraph of $\bar \rho$ in $N_n \doteq 2^n$ regions of measure $\ell_n \doteq 2^{-n}L$ as follows. Set
\begin{subequations}\label{eq:initial_ftl}
\begin{equation}
  \bar{x}_0^n \doteq \bar{x}_{\min},
\end{equation}
and recursively
\begin{align}
  &\bar{x}_i^n \doteq \sup\left\{x\in \reali \, \colon\,\int_{\bar{x}^n_{i-1}}^x\bar\rho(y) \,{\der}y<\ell_n\right\},&
  i=1,\ldots,N_n.
\end{align}
\end{subequations}
It is easily seen that $\bar{x}^n_{N_n}=\bar{x}_{\max}$, and $\bar{x}_{N_n-i}^n = \bar{x}_{N_{n+m}-2^mi}^{n+m}$ for all $i=0,\ldots,N_n$. Since we are always assuming that $\bar \rho \in \Ll\infty(\reali)$, let us set
\[R\doteq \norma{\bar \rho}_{\Ll\infty(\reali)}.\]
We have
\begin{align*}
  &\ell_n=\int_{\bar{x}_i^n}^{\bar{x}_{i+1}^n} \bar\rho(y) \,{\der}y \le \left (\bar{x}_{i+1}^n - \bar{x}_{i}^n\right ) R ,&
  i=0,\ldots,N_n-1.
\end{align*}
Thus the condition~\eqref{eq:FTL_assumption1} is satisfied with $\ell=\ell_n$, and we can take the values $\bar{x}_0^n,\ldots,\bar{x}_{N_n}^n$ as the initial positions of the $(N_n+1)$ particles in the $n$--depending version of the FTL model~\eqref{eq:FTL_intro}
\begin{subequations}\label{eq:ftl}
\begin{align}
    \label{eq:ftl1}
    &\dot x_{N_n}^n(t) = v_{\max} ,\\
    \label{eq:ftl2}
    &\dot x_i^n(t) = v \left(\frac{\ell_n}{x_{i+1}^n(t) - x_i^n(t)}\right) , && i=0, \ldots, N_n-1 ,\\
    \label{eq:ftl3}
    &x_i^n(0) = \bar{x}_i^n ,&& i=0, \ldots, N_n .
\end{align}
\end{subequations}
The existence of a global-in-time solution to~\eqref{eq:ftl} follows from Lemma~\ref{lem:1}. Moreover, from~\eqref{eq:ftl1} we immediately deduce that
\begin{align*}
    x_{N_n}^n(t) = \bar{x}_{\max} + v_{\max} \,t .
\end{align*}
Finally, since $v$ is decreasing, and its argument $\ell/[x_{i+1}^n(t)-x_i^n(t)]$ is always bounded above by $R$, we have
\[x_0^n(t)\geq \bar{x}_{\min} + v(R) \, t.\]
We stress once again that $v(R)$ may be negative.

By introducing in~\eqref{eq:ftl} the new variable
\begin{align}\label{eq:Deltan}
    &y^n_i(t) \doteq \frac{\ell_n}{x_{i+1}^n(t) - x_i^n(t)} ,&
    i=0,\ldots,N_n-1,
\end{align}
we obtain
\begin{subequations}\label{eq:dyi}
\begin{align}\label{eq:dyi1}
    &\dot{y}^n_{N-1} = -\frac{(y^n_{N-1})^2}{\ell_n} \left[v_{\max} - v(y^n_{N-1})\right],
    \\ \label{eq:dyi2}
    & \dot{y}^n_i = -\frac{(y^n_i)^2}{\ell_n} \left[v(y^n_{i+1})-v(y^n_i)\right], && i=0, \ldots, N_n-2 ,
    \\ \label{eq:dyi3}
    &y^n_i(0) = \bar{y}^n_i \doteq \frac{\ell_n}{\bar{x}_{i+1}^n - \bar{x}_i^n} ,&&
    i=0,\ldots,N_n-1.
\end{align}
\end{subequations}
Observe that $\ell_n/\left[\bar{x}_{\max} - \bar{x}_{\min} + (v_{\max}-v(R)) \, t\right] \le y^n_i(t) \le R$ for all $t\ge0$ in view of Lemma~\ref{lem:1}. The quantity $y^n_i$ can be seen as a discrete version of the density $\rho$ in Lagrangian coordinates, and the ODEs~\eqref{eq:dyi1}--\eqref{eq:dyi2} are a discrete Lagrangian version of the scalar conservation law~\eqref{eq:LWR_intro1}.

We are now ready to state the main result of this paper.
\begin{theorem}\label{thm:main}
    Let $\bar\rho$ satisfy the condition~(In) and $v$ the condition~(V1) and (V2). Assume further that either
    \begin{enumerate}
      \item[$\bullet$] $\bar\rho$ satisfies (InBV),
    \end{enumerate}
    or
    \begin{enumerate}
      \item[$\bullet$] $v$ satisfies (V3).
    \end{enumerate}
    Define the piecewise constant (with respect to $x$) density
    \begin{align}\label{eq:hrn}
        &\hat{\rho}^n(t,x) \doteq
        \sum_{i=0}^{N_n-1} y^n_i(t) \, \caratt{\left[x_i^n(t), x_{i+1}^n(t)\right[}(x) ,
    \end{align}
    and the empirical measure
    \begin{align}\label{eq:trn}
        &\tilde{\rho}^n(t,x) \doteq \ell_n \sum_{i=0}^{N_n-1} \displaystyle\dirac{\strut{\textstyle x_i^n(t)}}(x).
    \end{align}
    Then the sequence $(\hat{\rho}^n)_{n\in\N}$ converges to the unique entropy solution $\rho$ of the Cauchy problem~\eqref{eq:LWRmodel} almost everywhere and in $\Llloc1\left(\left[0,+\infty\right[\times \reali\right)$. Moreover, the sequence $(\tilde{\rho}^n)_{n\in\N}$ converges to $\rho$ in the topology of $\Llloc1\left(\left[0,+\infty\right[;\; d_{L,1}\right)$.
\end{theorem}

\section{Proof of the main result}\label{sec:result}

Our strategy for the proof of Theorem~\ref{thm:main} can be resumed as follows:
\begin{enumerate}[(i)]
  \item Following the notation introduced in Section~\ref{sec:preliminaries}, we set $\hat{F}^n=F_{\hat{\rho}^n}$ and $\hat{X}^n=\mathcal{X}[\hat{F}^n]$, respectively $\tilde{F}^n=F_{\tilde{\rho}^n}$ and $\tilde{X}^n=\mathcal{X}[\tilde{F}^n]$, as the cumulative distribution of $\hat{\rho}^n$, respectively $\tilde{\rho}^n$, and its pseudo inverses. Introduce the \emph{discrete Lagrangian density} $$\check\rho^n =\hat{\rho}^n \circ \hat{X}^n.$$
  \item As a first step we prove that the sequence of piecewise constant pseudo-inverse distributions $(\tilde{X}^n)_{n\in\N}$ has a strong limit $X$ in $\Llloc1([0,+\infty[\times [0,L];\!\reali)$, which is equivalent to having $(\tilde{\rho}^n)_{n\in\N}$ converging to a measure $\rho$ in the $\Llloc1([0,+\infty[;\,d_{L,1})$ topology. At the same time, we shall also prove that $(\hat{X}^n)_{n\in\N}$ converges strongly in $\Llloc1([0,+\infty[\times [0,L];\reali)$ to the same limit $X$, i.e.~$(\hat{\rho}^n)_{n\in\N}$ converges to $\rho$ in $\Llloc1([0,+\infty[;\,d_{L,1})$.
  \item We then prove that the limit pseudo-inverse function $X$ has difference quotients bounded below by~$1/R$. This fact allows to prove that the limit measure $\rho$ in (ii) is actually in $\Ll\infty$ and is a.e.~bounded by~$R$. At the same time, we easily infer weak--$*$ convergence of $(\check{\rho}^n)_{n\in\N}$ to a limit $\check{\rho}$ in $\Ll\infty$. It remains to prove that $\check{\rho}\circ F = \rho$, and that such limit is the unique entropy solution to~\eqref{eq:LWRmodel}. This requires stronger estimates on $\hat{\rho}^n$.
  \item A direct proof of a uniform $\BV$ estimate for $\hat{\rho}^n$ can be performed in the case of $\mathcal{M}_L \cap \BV$ initial datum. In the case of general $\mathcal{M}_L \cap \Ll\infty$ initial datum, and with $v$ satisfying (V3), we shall prove that the discrete Lagrangian density $\check{\rho}^n$ satisfies a (uniform) \emph{discrete version of the Oleinik condition}, which implies automatically a $\BV$ uniform estimate for $\check{\rho}^n$, and hence for $\hat{\rho}^n$.
  \item The definition of weak solution~\eqref{eq:weaksol} for $\rho$ follows from the $n\rightarrow +\infty$ limit of the formulation of~\eqref{eq:ftl} as a PDE
  \begin{equation}\label{eq:PDEn}
        \tilde{X}_t^n = v(\check{\rho}^n) .
  \end{equation}
  \item We finally recover the entropy condition~\eqref{eq:entropy_ineq} in the discrete setting, and use the strong $\Ll1$ compactness to pass it to the limit.
\end{enumerate}

\subsection{Weak convergence of the approximating scheme}\label{sec:convergence}

Throughout this subsection we shall assume that $v$ satisfies (V1) and (V2). Let $\hat{\rho}^n$ and $\tilde{\rho}^n$ be defined as in~\eqref{eq:hrn} and~\eqref{eq:trn} respectively. We have that $\hat{\rho}^n(t), \tilde{\rho}^n(t) \in \mathcal{M}_L$ for all $t\ge0$. Thus we can consider the cumulative distributions associated to $\hat{\rho}^n$ and $\tilde{\rho}^n$ (recall that $\tilde{\rho}^n$ is an empirical measure)
\begin{align*}
    &\hat{F}^n(t,x) \doteq \int_{-\infty}^x \hat{\rho}^n(t,y) \,{\der}y,&
    \tilde{F}^n(t,x) \doteq \tilde{\rho}^n(]-\infty,x]),
\end{align*}
and their pseudo-inverses
\begin{align*}
    &\hat{X}^n \doteq \mathcal{X}\left[\hat{F}^n\right],&
    \tilde{X}^n \doteq \mathcal{X}\left[\tilde{F}^n\right],
\end{align*}
extended to $z=L$ by taking $\hat{X}^n(t,L) = x_{N_n}^n(t) = \tilde{X}^n(t,L)$.
\begin{figure}[htpb]
      \centering\begin{psfrags}\tiny
      \psfrag{r}[r,c]{$\hat{\rho}$}
      \psfrag{F}[r,c]{$\hat F$}
      \psfrag{X}[r,c]{$\hat X$}
      \psfrag{1}[r,t]{$R$}
      \psfrag{F8}[r,c]{$L$}
      \psfrag{F7}[r,c]{$7L/8$}
      \psfrag{F6}[r,c]{$3L/4$}
      \psfrag{F5}[r,c]{$5L/8$}
      \psfrag{F4}[r,c]{$L/2$}
      \psfrag{F3}[r,c]{$3L/8$}
      \psfrag{F2}[r,c]{$L/4$}
      \psfrag{F1}[r,c]{$L/8$}
      \psfrag{x0}[l,B]{$x_{\!0}^{\,}$}
      \psfrag{x1}[l,B]{$x_{\!1}^{\,}$}
      \psfrag{x2}[l,B]{$x_{\!2}^{\,}$}
      \psfrag{x3}[l,B]{$x_{\!3}^{\,}$}
      \psfrag{x4}[l,B]{$\,x_{\!4}^{\,}$}
      \psfrag{x5}[l,B]{$x_{\!5}^{\,}$}
      \psfrag{x6}[l,B]{$x_{\!6}^{\,}$}
      \psfrag{x7}[l,B]{$x_{\!7}^{\,}$}
      \psfrag{x8}[l,B]{$x_{\!8}^{\,}$}
      \psfrag{x}[l,B]{$x$}
        \includegraphics[width=.3\textwidth]{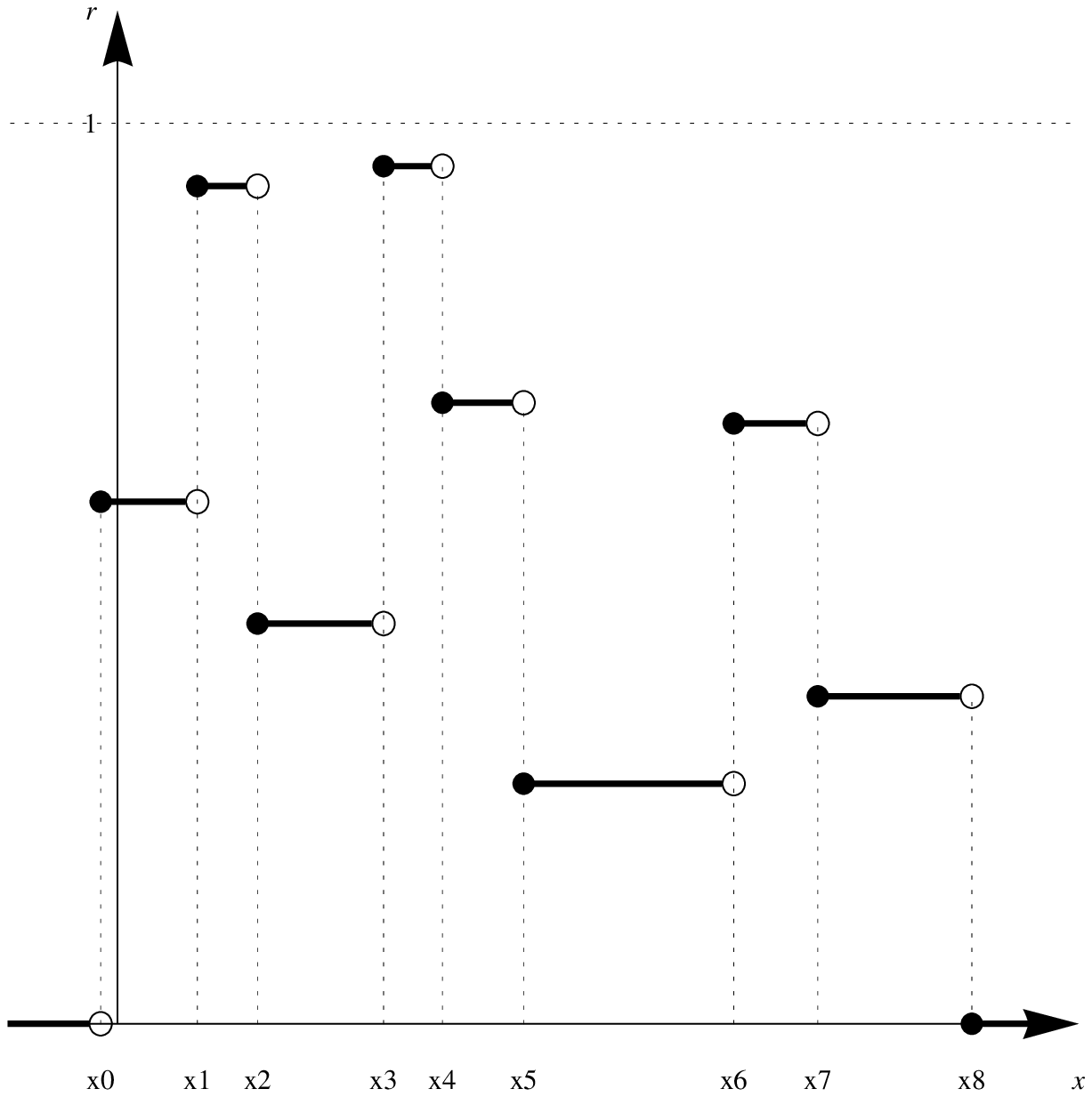}\qquad
        \includegraphics[width=.3\textwidth]{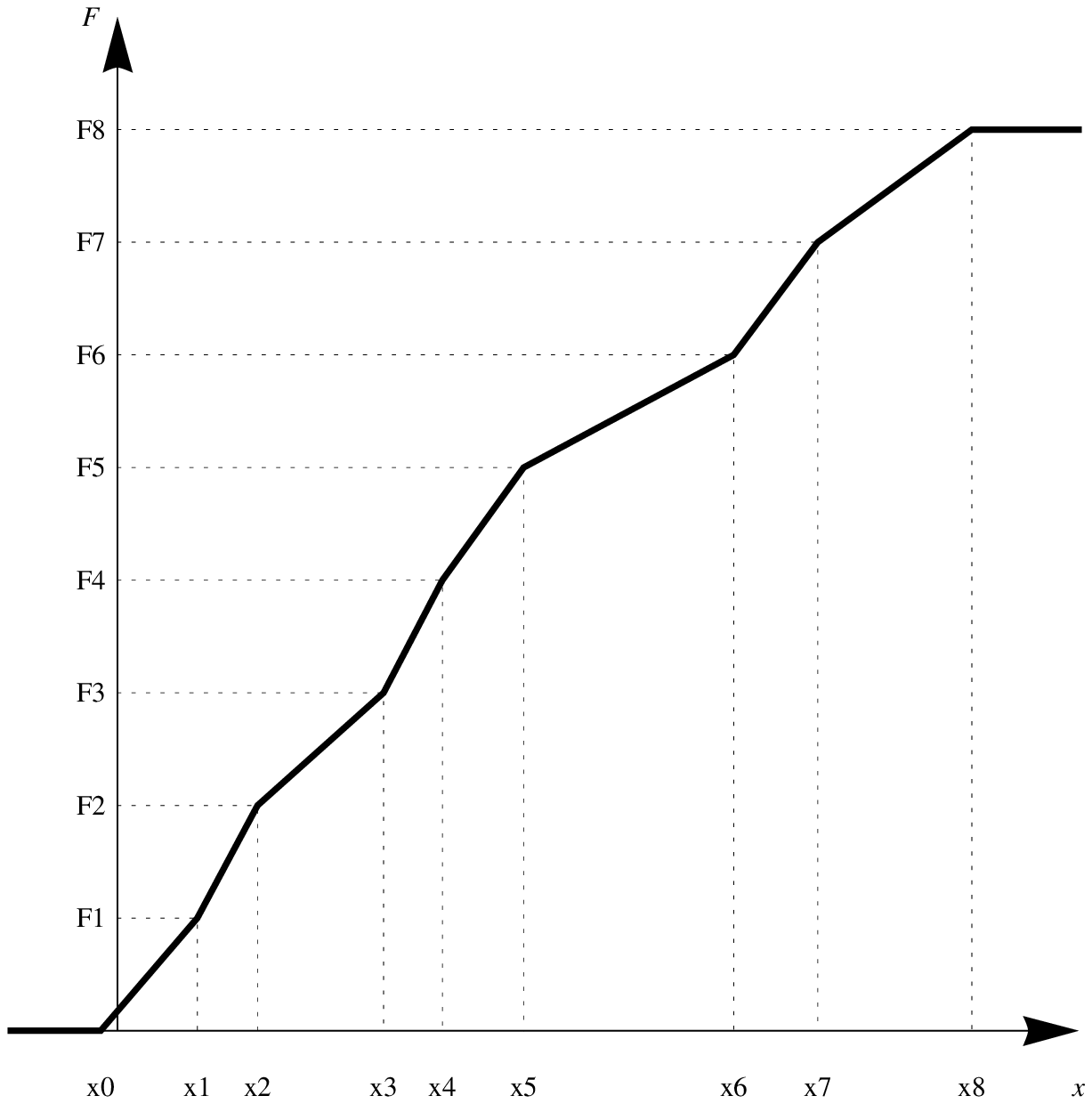}\qquad
      \psfrag{z}[l,t]{$z$}
      \psfrag{F8}[l,t]{$L$}
      \psfrag{F7}[l,t]{$\frac{7L}{8}$}
      \psfrag{F6}[l,t]{$\frac{3L}{4}$}
      \psfrag{F5}[l,t]{$\frac{5L}{8}$}
      \psfrag{F4}[l,t]{$\frac{L}{2}$}
      \psfrag{F3}[l,t]{$\frac{3L}{8}$}
      \psfrag{F2}[l,t]{$\frac{L}{4}$}
      \psfrag{F1}[l,t]{$\frac{L}{8}$}
      \psfrag{x0}[r,c]{$x_0$}
      \psfrag{x1}[r,c]{$x_1$}
      \psfrag{x2}[r,c]{$x_2$}
      \psfrag{x3}[r,c]{$x_3$}
      \psfrag{x4}[r,c]{$x_4$}
      \psfrag{x5}[r,c]{$x_5$}
      \psfrag{x6}[r,c]{$x_6$}
      \psfrag{x7}[r,c]{$x_7$}
      \psfrag{x8}[r,c]{$x_8$}
        \includegraphics[width=.3\textwidth]{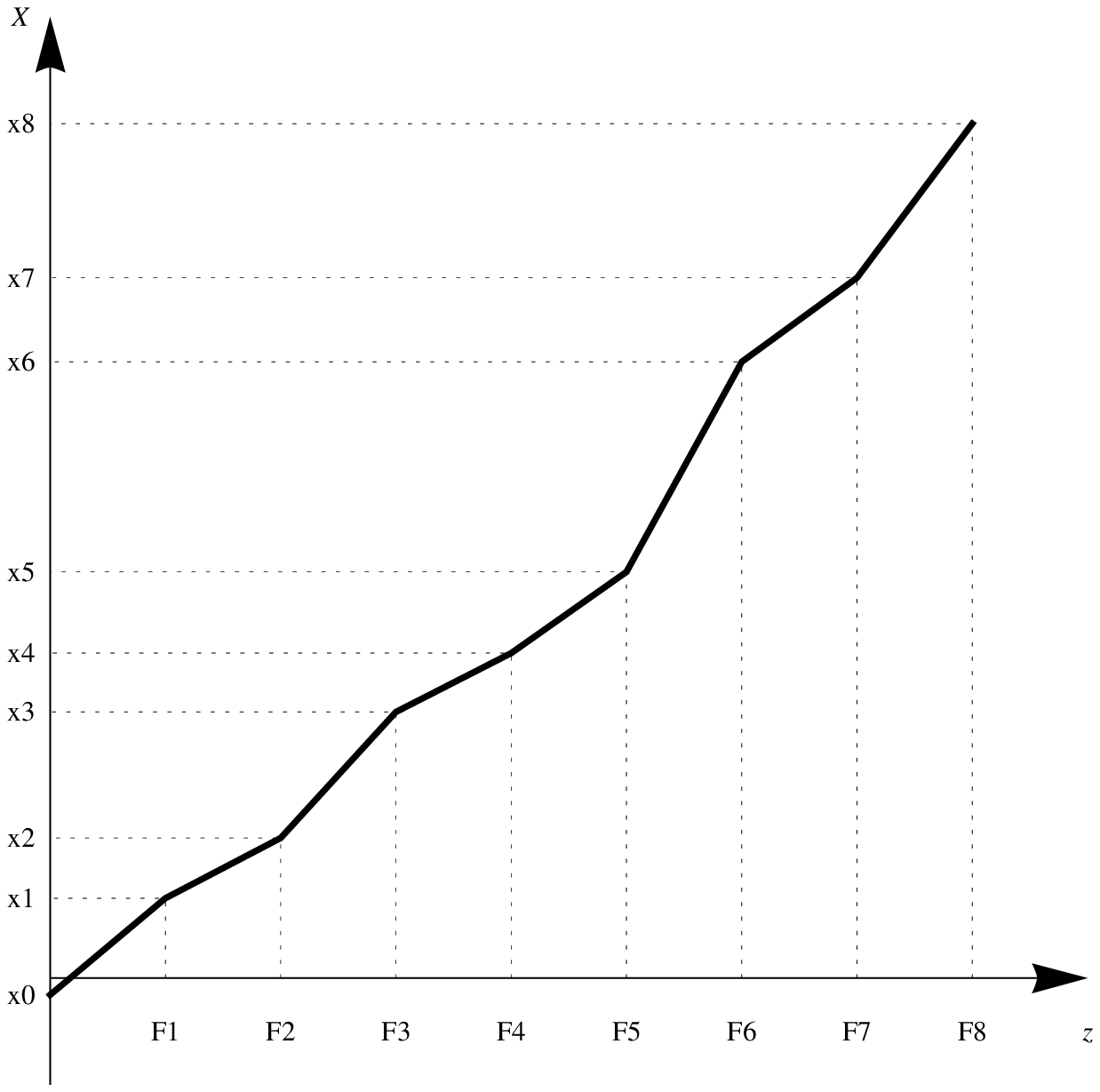}
      \end{psfrags}
      \caption{Maps of the form, respectively from the left, \eqref{eq:hrn}, \eqref{eq:hFn} and~\eqref{eq:hXn} with $n(=3)$ and $t(\ge0)$ omitted.}
\label{fig:Korn1}
\end{figure}
\begin{figure}[htpb]
      \centering\begin{psfrags}\tiny
      \psfrag{r}[r,c]{$\tilde\rho$}
      \psfrag{F}[r,c]{$\tilde F$}
      \psfrag{X}[r,c]{$\tilde{X}$}
      \psfrag{1}[r,t]{$1$}
      \psfrag{F8}[r,c]{$L$}
      \psfrag{F7}[r,c]{$7L/8$}
      \psfrag{F6}[r,c]{$3L/4$}
      \psfrag{F5}[r,c]{$5L/8$}
      \psfrag{F4}[r,c]{$L/2$}
      \psfrag{F3}[r,c]{$3L/8$}
      \psfrag{F2}[r,c]{$L/4$}
      \psfrag{F1}[r,c]{$L/8$}
      \psfrag{x0}[l,B]{$x_{\!0}^{\,}$}
      \psfrag{x1}[l,B]{$x_{\!1}^{\,}$}
      \psfrag{x2}[l,B]{$x_{\!2}^{\,}$}
      \psfrag{x3}[l,B]{$x_{\!3}^{\,}$}
      \psfrag{x4}[l,B]{$\,x_{\!4}^{\,}$}
      \psfrag{x5}[l,B]{$x_{\!5}^{\,}$}
      \psfrag{x6}[l,B]{$x_{\!6}^{\,}$}
      \psfrag{x7}[l,B]{$x_{\!7}^{\,}$}
      \psfrag{x8}[l,B]{$x_{\!8}^{\,}$}
      \psfrag{x}[l,B]{$x$}
        \includegraphics[width=.3\textwidth]{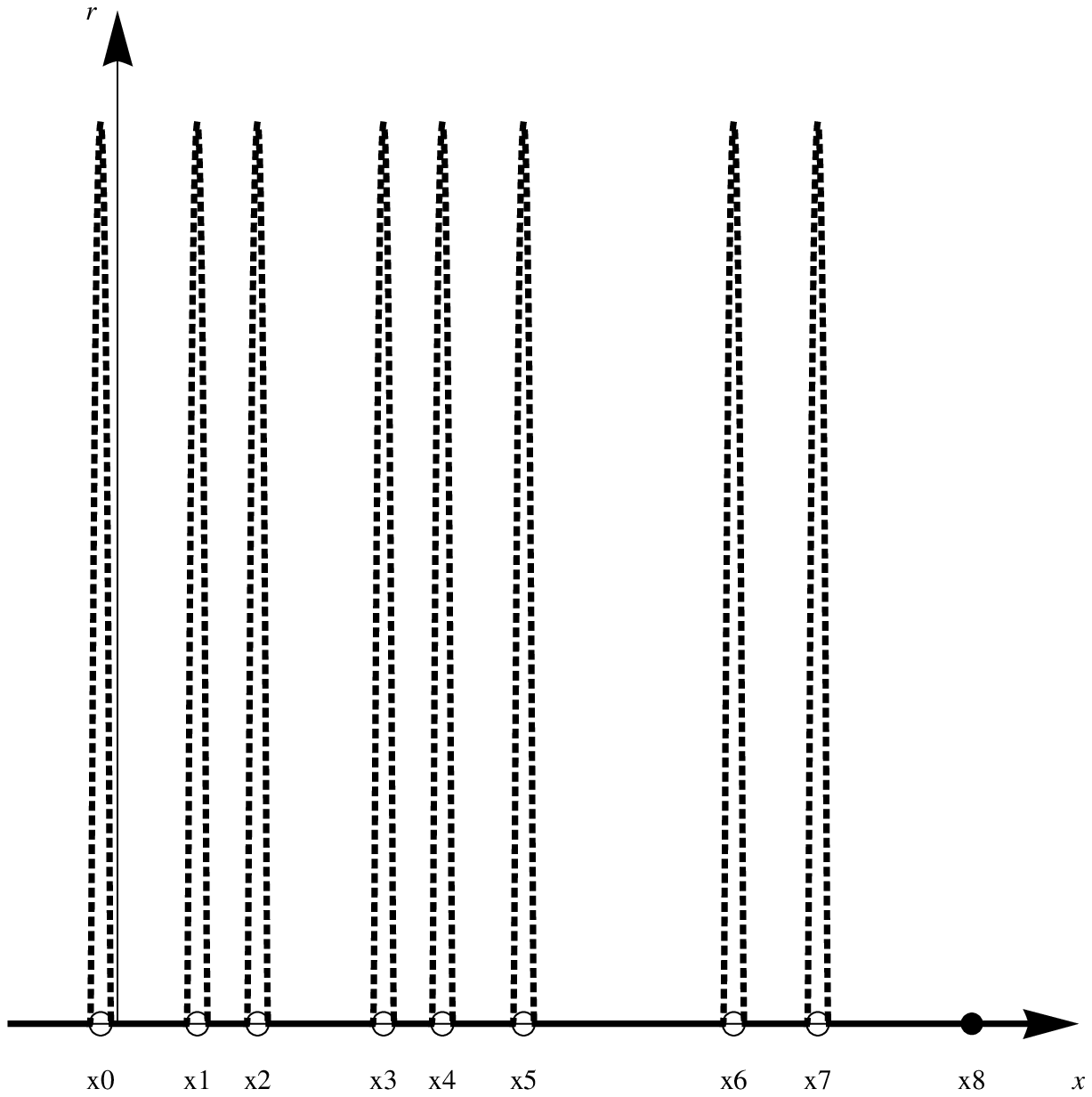}\qquad
        \includegraphics[width=.3\textwidth]{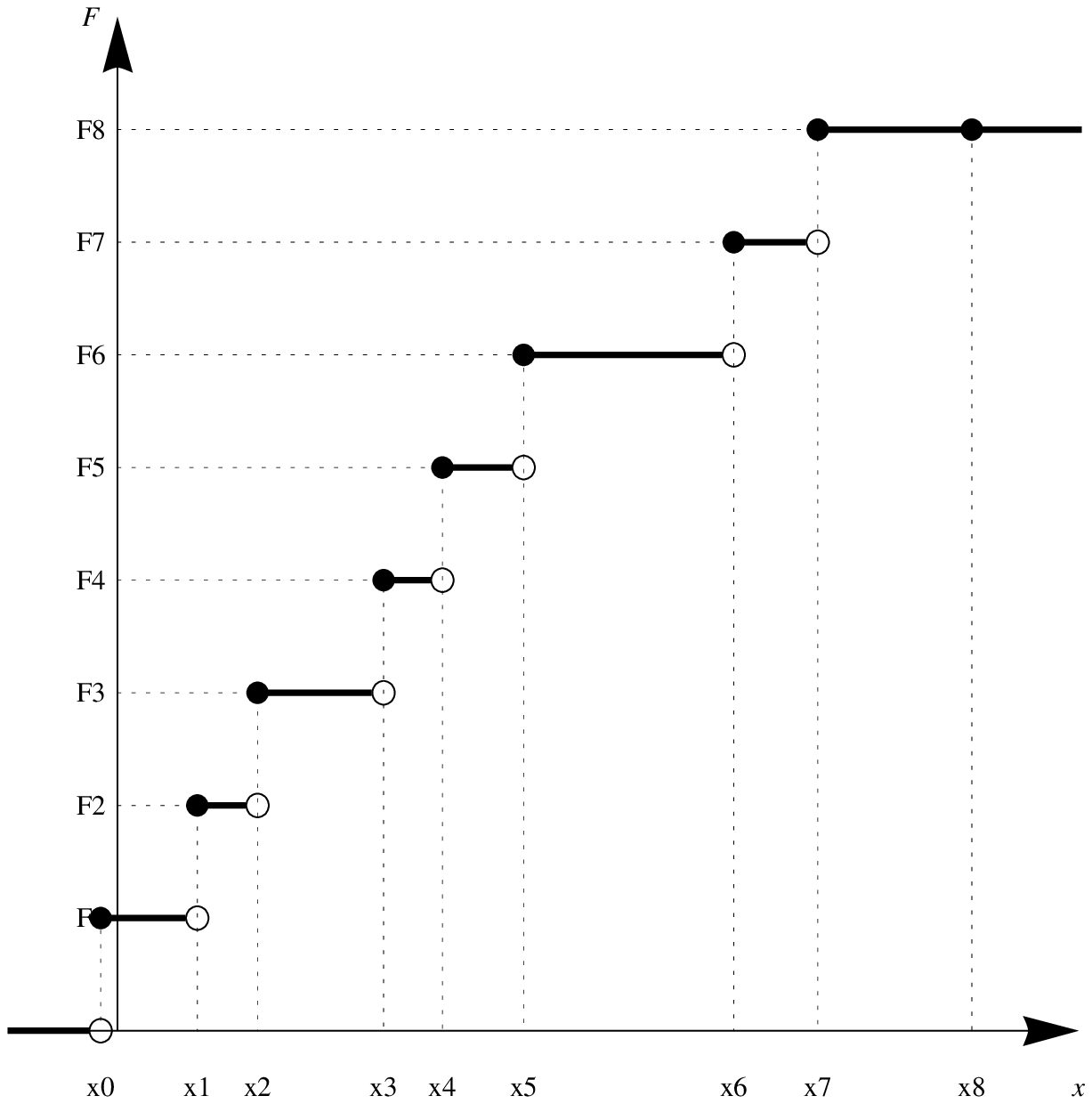}\qquad
      \psfrag{z}[l,t]{$z$}
      \psfrag{F8}[l,t]{$L$}
      \psfrag{F7}[l,t]{$\frac{7L}{8}$}
      \psfrag{F6}[l,t]{$\frac{3L}{4}$}
      \psfrag{F5}[l,t]{$\frac{5L}{8}$}
      \psfrag{F4}[l,t]{$\frac{L}{2}$}
      \psfrag{F3}[l,t]{$\frac{3L}{8}$}
      \psfrag{F2}[l,t]{$\frac{L}{4}$}
      \psfrag{F1}[l,t]{$\frac{L}{8}$}
      \psfrag{x0}[r,c]{$x_0$}
      \psfrag{x1}[r,c]{$x_1$}
      \psfrag{x2}[r,c]{$x_2$}
      \psfrag{x3}[r,c]{$x_3$}
      \psfrag{x4}[r,c]{$x_4$}
      \psfrag{x5}[r,c]{$x_5$}
      \psfrag{x6}[r,c]{$x_6$}
      \psfrag{x7}[r,c]{$x_7$}
      \psfrag{x8}[r,c]{$x_8$}
        \includegraphics[width=.3\textwidth]{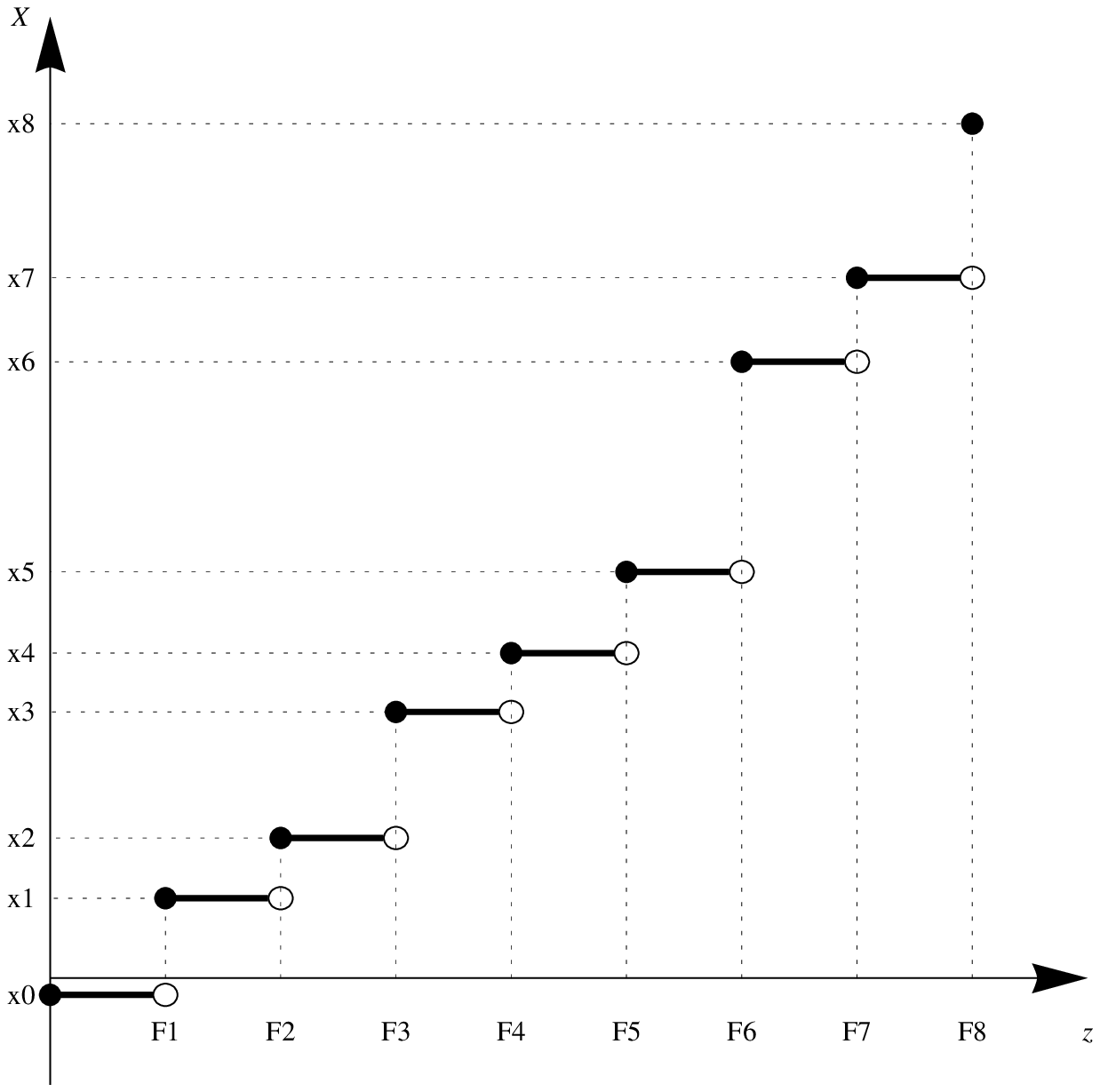}
      \end{psfrags}
      \caption{Maps of the form, respectively from the left, \eqref{eq:trn}, \eqref{eq:tFn} and~\eqref{eq:tXn} with $n(=3)$ and $t(\ge0)$ omitted.}
\label{fig:Korn2}
\end{figure}
By definition, see figures~\ref{fig:Korn1} and~\ref{fig:Korn2}, for all $t \ge 0$, $z \in [0,L]$ and $x \in \reali$ we have
\begin{align}\nonumber
    \hat{F}^n(t,x) =&\, \sum_{i=0}^{N_n-1} \left[\vphantom{\frac{z-i\,\ell_n}{y^n_i(t)}} i \, \ell_n + y^n_i(t) \left[x-x_i^n(t)\right] \right] \caratt{\left[x_i^n(t), x_{i+1}^n(t)\right[}(x)
    \\\label{eq:hFn}
    &\, + L \,\caratt{\left[x_{N_n}^n(t), +\infty\right[}(x) ,
    \\ \nonumber
    \hat{X}^n(t,z) =&\, \sum_{i=0}^{N_n-2} \left[x_i^n(t) + \frac{z-i\,\ell_n}{y^n_i(t)}\right] \caratt{\left[i\,\ell_n, (i+1)\,\ell_n\right[}(z)
    \\\label{eq:hXn}
    &\, +\left[x_{N_n-1}^n(t) + \frac{z-L+\ell_n}{y^n_{N_n-1}(t)}\right] \caratt{\left[L-\ell_n, L\right]}(z),
    \\
    \label{eq:tFn}
    \tilde{F}^n(t,x) =&\, \sum_{i=0}^{N_n-2} \ell_n \, \left(i+1\right)  \caratt{\left[x_i^n(t), x_{i+1}^n(t)\right[}(x)
    + L \,\caratt{\left[x_{N_n-1}^n(t), +\infty\right[}(x) ,
    \\
    \label{eq:tXn}
    \tilde{X}^n(t,z) =&\, \sum_{i=0}^{N_n-1} x_i^n(t) \, \caratt{\left[i \,\ell_n, (i+1) \,\ell_n\right[}(z)
    + x_{N_n}^n(t) \, \caratt{\{L\}}(z).
\end{align}
Observe that for any fixed $t\ge0$
\begin{itemize}
  \item[$\bullet$] $z \mapsto \hat{X}^n(t,z)$ and $x \mapsto \hat{F}^n(t,x)$ are piecewise linear continuous and non-decreasing;
  \item[$\bullet$] $\hat{F}^n(t) \,\colon\, [x_0^n(t), x_{N_n}^n(t)] \to [0,L]$ and $\hat{X}^n(t) \,\colon\, [0,L] \to [x_0^n(t), x_{N_n}^n(t)]$ are strictly increasing and are inverse functions of  each other in the classical sense;
  \item[$\bullet$] $z \mapsto \tilde{X}^n(t,z)$ and $x \mapsto \tilde{F}^n(t,x)$ are piecewise constant with $N_n$ jumps of discontinuity, right continuous and non-decreasing;
  \item[$\bullet$] $\hat{F}^n(t,x) \le \tilde{F}^n(t,x)$ for any $x \in \reali$ and $\tilde{X}^n(t,z) \le \hat{X}^n(t,z)$ for any $z \in [0,L]$;
  \item[$\bullet$] $\tilde{F}^{n+1}(t,x) \le \tilde{F}^n(t,x)$ for any $x \in \reali$ and $\tilde{X}^n(t,z) \le \tilde{X}^{n+1}(t,z)$ for any $z \in [0,L]$;
  \item[$\bullet$] $\hat{\rho}^n(t,x) = \hat{F}^n_x(t,x)$ for all $x\ne x_i^n(t)$, $i=1,\ldots,N_n$, while $\tilde{\rho}^n = \tilde{F}^n_x$ in the sense of distributions.
\end{itemize}
\begin{figure}[htpb]
      \centering\begin{psfrags}\tiny
      \psfrag{r}[r,c]{$\check{\rho}$}
      \psfrag{1}[r,c]{$1$}
      \psfrag{z}[l,t]{$z$}
      \psfrag{F8}[l,t]{$L$}
      \psfrag{F7}[l,t]{$\frac{7L}{8}$}
      \psfrag{F6}[l,t]{$\frac{3L}{4}$}
      \psfrag{F5}[l,t]{$\frac{5L}{8}$}
      \psfrag{F4}[l,t]{$\frac{L}{2}$}
      \psfrag{F3}[l,t]{$\frac{3L}{8}$}
      \psfrag{F2}[l,t]{$\frac{L}{4}$}
      \psfrag{F1}[l,t]{$\frac{L}{8}$}
        \includegraphics[width=.3\textwidth]{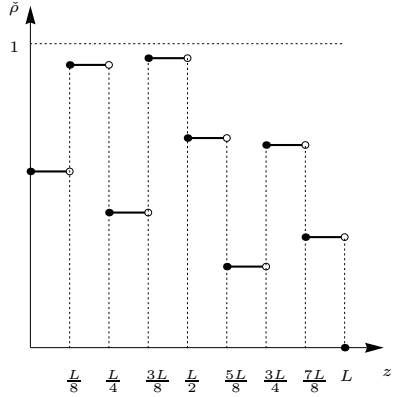}
      \end{psfrags}
      \caption{Map of the form~\eqref{eq:crn} with $n(=3)$ and $t(\ge0)$ omitted.}
\label{fig:Korn3}
\end{figure}
For later use, see Figure~\ref{fig:Korn3}, we introduce also the \emph{discrete Lagrangian density}
\begin{align}\label{eq:crn}
    &\check{\rho}^n (t,z) \doteq \hat{\rho}^n\left(t, \hat{X}^n(t,z)\right)
    = \sum_{i=0}^{N_n-1} y^n_i(t) \, \caratt{\left[i\,\ell_n, (i+1)\,\ell_n\right[}(z)
\end{align}
and observe that
\begin{align}\label{eq:approx_PDE}
    &\tilde{X}_t^n (t,z) = v\left(\check{\rho}^n(t,z)\right),&&t>0,~ z \in[0,L].
\end{align}
As a first step, we want to prove that $(\tilde{X}^n)_{n\in\N}$ and $(\hat{X}^n)_{n\in\N}$ have the same unique limit in $\Llloc1([0,+\infty[\times [0,L];\reali)$.
\begin{proposition}[Definition of $X$]\label{pro:convergence1}
    There exists a unique non-decreasing and $z$--right continuous function $X =X(t,z)\in \Ll\infty\left(\left[0,+\infty\right[ \times \left[0,L\right]; \reali\right)$, such that
    \begin{align*}
        &(\hat{X}^n)_{n\in\N} \hbox{ and } (\tilde{X}^n)_{n\in\N} \hbox{ converge to } X
        \hbox{ in } \Llloc1\left(\left[0,+\infty\right[ \times \left[0 ,L\right];\reali\right),
    \end{align*}
    and for any $t,s>0$
    \begin{subequations}\label{eq:Xestimates}
    \begin{align}
        &\tv\left[X(t)\right] \le \bar{x}_{\max} - \bar{x}_{\min} + (v_{\max}-v(R))\,t ,\\
        &\norma{X(t)}_{\Ll\infty([0,L];\reali)} \le \max\left\{\modulo{\bar{x}_{\min}+v(R)\,t}, \modulo{\bar{x}_{\max}+ v_{\max}\,t}\right\} ,\\
        &\int_0^L \modulo{X(t,z) - X(s,z)} \,{\der}z \le \max\{\modulo{v_{\max}},\modulo{v(R)}\} \,L \, \modulo{t-s} .
    \end{align}
    \end{subequations}
    Moreover, $(\tilde{X}^n)_{n\in\N}$ converges to $X$ a.e.~on $\left[0,+\infty\right[ \times \left[0 ,L\right]$.
\end{proposition}

\begin{proof}
   Fix $T > 0$, and let $n > 0$.

    \noindent$\bullet$~{\textbf{\textsc{Step~1: $\bf\tilde{X}^n \to X$.}}}
    Since $z \mapsto \tilde{X}^n(t,z)$ is non-decreasing with $\tilde{X}^n(t,0) = x_0^n(t) \ge \bar{x}_0^{n}+v(R) \, t = \bar{x}_{\min}+v(R) \, t $ and $\tilde{X}^n(t,L) = \bar{x}_{\max} + v_{\max} \, t$, we have that
    \begin{align*}
        &\tv\left[\tilde{X}^n(t)\right] \le \bar{x}_{\max} - \bar{x}_{\min} + (v_{\max}-v(R))\,t ,\\
        &\norma{\tilde{X}^n(t)}_{\Ll\infty([0,L];\reali)} \le \max\{\modulo{\bar{x}_{\min}+v(R) \, t}, \modulo{\bar{x}_{\max}+ v_{\max}\,t}\} .
    \end{align*}
    Moreover, if $s < t$, then by~\eqref{eq:ftl2} and~\eqref{eq:tXn}
    \begin{align*}
        &\int_0^L \modulo{\tilde{X}^n(t,z) - \tilde{X}^n(s,z)} \,{\der}z =
        \sum_{i=0}^{N_n-1} \ell_n \left|x_i^n(t) - x_i^n(s)\right|
        \\
        &\le\sum_{i=0}^{N_n-1} \ell_n \int_s^t \left|v\left(y^n_i(\tau)\right)\right|  {\der}\tau \le \max\{\modulo{v_{\max}},\modulo{v(R)}\} \,L \left(t-s\right) .
    \end{align*}
    Therefore, by applying Helly's theorem in the form~\cite[Theorem~2.4]{BressanBook}, up to a subsequence, $(\tilde{X}^n)_{n\in\N}$ converges in $\Llloc1\left(\left[0,+\infty\right[ \times \left[0 ,L\right];\reali\right)$ to a function $X$ which is right continuous w.r.t.~$z$ and satisfying~\eqref{eq:Xestimates}. Finally, since $\tilde{X}^{n+1}(t,z) \le \tilde{X}^{n}(t,z)$ for all $t\ge0$ and $z \in [0,L]$, the whole sequence $(\tilde{X}^n)_{n\in\N}$ converges to $X$ and a.e.~on $\left[0,+\infty\right[ \times \left[0 ,L\right]$.

    \noindent$\bullet$~{\textbf{\textsc{Step~2: $\bf\hat{X}^n \to X$.}}}
    By definition, see~\eqref{eq:Deltan}, \eqref{eq:hXn} and~\eqref{eq:tXn}, we have for all $t \in \left[0,T\right]$
    \begin{align*}
        &\int_0^L \modulo{\hat{X}^n(t,z) - \tilde{X}^n(t,z)} \,{\der}z =
        \sum_{i=0}^{N_n-1} y^n_i(t)^{-1} \int_{i\,\ell_n}^{(i+1)\,\ell_n} \left[z-i\,\ell_n\right]\,{\der}z
        \\
        &=
        \frac{\ell_n}{2} \sum_{i=0}^{N_n-1} \left[x_{i+1}^n(t) - x_i^n(t)\right] =
        \frac{\ell_n}{2} \left[x_{N_n}^n(t) - x_0^n(t)\right] \\
        & \le
        \frac{\ell_n}{2} \left[\vphantom{x_{N_n}^n} \bar{x}_{\max} - \bar{x}_{\min} +(v_{\max}-v(R)) \,T\right] ,
    \end{align*}
    and the proof is complete as $(\tilde{X}^n)_{n\in\N}$ converges to $X$ in view of \textsc{Step~1}.
\end{proof}

In the next lemma we prove that $X$ inherits the maximum principle property satisfied by $\tilde{X}^n$ proven in Lemma~\ref{lem:1}.
\begin{lemma}\label{lem:2}
    For all $t\ge0$ and for a.e.~$z_1, z_2 \in \left[0,L\right]$ with $z_1<z_2$ we have
    \begin{equation}\label{eq:XLipwrtz}
        \frac{z_2-z_1}{R} \le X(t,z_2) - X(t,z_1) \le \bar{x}_{\max} - \bar{x}_{\min} +(v_{\max}-v(R)) \, t .
    \end{equation}
\end{lemma}
\begin{proof}
    The upper bound is obvious. Take $0\le z_1 < z_2 \le L$. For $n>0$ sufficiently large, we can take $i,j \in \{0,1,\ldots,N_n\}$ such that $i<j$, $i \,\ell_n \le z_1 < (i+1)\,\ell_n$ and $\ell_n  \,j \le z_2 < \ell_n \,(j+1)$. By~\eqref{eq:tXn} and Lemma~\ref{lem:1} we have
    \begin{align*}
        &\frac{\tilde{X}^n(t,z_2) - \tilde{X}^n(t,z_1)}{z_2-z_1} \ge
        \frac{x_j^n(t) - x_i^n(t)}{(j+1) \,\ell_n - i\,\ell_n} \ge
        \frac{(j-i) \,\ell_n}{R\left((j+1) \,\ell_n - i\,\ell_n\right)}
        \\ &=
        \frac{1}{R} - \frac{1}{R\left(j-i+1\right)} \ge
        \frac{1}{R} - \frac{1}{R\left((z_2 \, \ell_n^{-1} - 1) - z_1 \, \ell_n^{-1} + 1\right)} =
        \frac{1}{R} - \frac{\ell_n}{R(z_2 - z_1)} .
    \end{align*}
    By letting $n$ go to infinity in the above estimate we conclude the proof. Indeed, as $n$ goes to infinity we have that $\ell_n/[R(z_2 - z_1)]$ converges to zero and $(\tilde{X}^n)_{n\in\N}$ converges to $X$ a.e.~on $\left[0,+\infty\right[ \times \left[0 ,L\right]$ in view of Proposition~\ref{pro:convergence1}.
\end{proof}

\begin{proposition}[Definition of $F$]\label{pro:convergence2}
    $(\hat{F}^n)_{n\in\N}$ and $(\tilde{F}^n)_{n\in\N}$ converge to $F \doteq \mathcal{F}[X]$ in $\Llloc1 (\left[0,+\infty\right[ \times \reali;[0,L])$. Moreover, $(\tilde{F}^n)_{n\in\N}$ converges to $F$ a.e.~on $\left[0,+\infty\right[ \times \reali$.
\end{proposition}
\begin{proof}
    We first observe that by Lemma~\ref{lem:2} for any fixed $t\ge0$, the map $z \mapsto X(t,z)$ is strictly increasing and for all $z \in [0,L]$
    \begin{equation*}
        \frac{z}{R} + \bar{x}_{\min}+v(R)\, t \le X(t,z) \le \bar{x}_{\max} + v_{\max} \, t - \frac{L-z}{R} .
    \end{equation*}
    Thus, $F$ is well defined. The convergence of $(\hat{F}^n)_{n\in\N}$ and $(\tilde{F}^n)_{n\in\N}$ to $F$ follows from the basic property~\eqref{eq:wass_equiv0} of the scaled Wasserstein distance and from Proposition~\ref{pro:convergence1}. Indeed, for any $T>0$ we have
    \begin{align*}
        &\lim_{n\rightarrow +\infty} \norma{\hat{F}^n - F}_{\Ll1([0,T]\times\reali;\reali)}
        =\lim_{n\rightarrow +\infty} \norma{\hat{X}^n - X}_{\Ll1([0,T]\times[0,L];\reali)}
        =0 ,
        \\
        &\lim_{n\rightarrow +\infty} \norma{\tilde{F}^n - F}_{\Ll1([0,T]\times\reali;\reali)}
        =\lim_{n\rightarrow +\infty} \norma{\tilde{X}^n - X}_{\Ll1([0,T]\times[0,L];\reali)}
        =0 .
    \end{align*}
    Finally, $(\tilde{F}^n)_{n\in\N}$ converges to $F$ a.e.~on $\left[0,+\infty\right[ \times \reali$ because $\tilde{F}^{n+1}(t,x) \le \tilde{F}^n(t,x)$ for all $t\ge0$ and $x \in \reali$.
\end{proof}

\begin{lemma}\label{lem:3}
    For all $t\ge0$ and for a.e.~$x_1, x_2 \in \reali$ with $x_1 < x_2$ we have
    \begin{equation}\label{eq:FLipwrtz}
        0 \le F(t,x_2) - F(t,x_1) \le R(x_2-x_1) .
    \end{equation}
\end{lemma}
\begin{proof}
    Fix $x_1<x_2$ and denote $z_1 = F(t,x_1) \le z_2 = F(t,x_2)$. Since the lower bound is obvious, it is sufficient to prove that
    \begin{equation*}
        z_2 - z_1 \le R(x_2 - x_1) .
    \end{equation*}
    If $z_1 = z_2$, then there is nothing to prove. Assume therefore that $z_1 \ne z_2$ and fix $\eta \in \left]0, z_2 - z_1\right[$. By definition, $X(t,z) = \mathcal{X}[F](t,z) = \inf \{ x\in\reali \,\colon\, F(t,x) > z \}$. Since $F(t,x_2) = z_2 > z_2 -\eta$, we have that $X(t, z_2 - \eta) \le x_2$. Moreover, $X(t,z_1) \ge x_1$ because $z \mapsto X(t,z)$ is strictly increasing and right continuous. Therefore, by Lemma~\ref{lem:2} we have
    \begin{align*}
        x_2 - x_1 \ge
        X(t, z_2 - \eta) - X(t, z_1) \ge
       \frac{ z_2 - \eta - z_1}{R} .
    \end{align*}
    Since $\eta>0$ is arbitrary, we have $z_2 - z_1 \le R(x_2 - x_1)$.
\end{proof}

\begin{proposition}[Definition of $\rho$]\label{pro:convergence3}
    For any $t\ge0$, let $\rho(t)$ be the distributional derivative of $x \mapsto F(t,x)$, with $F$ defined in Lemma~\ref{lem:2}. Then:

    \noindent$\bullet$~$\rho(t,\cdot) \in \mathcal{M}_L$ for all $t\ge 0$,

    \noindent$\bullet$~$0 \le \rho(t,x) \le R$ for a.e.~$t\ge0$ and $x\in\reali$,

    \noindent$\bullet$~$(\tilde{\rho}^n)_{n\in\N}$ and $(\hat{\rho}^n)_{n\in\N}$ converge to $\rho$ in the topology of $\Llloc1\left(\left[0,+\infty\right[;\; d_{L,1}\right)$,

\end{proposition}
\begin{proof}
    For any fixed $t\ge0$, by Lemma~\ref{lem:3} we have that $x \mapsto F(t,x)$ is a Lipschitz function with $\lip\left(F(t)\right) \le R$. Hence its weak derivative $\rho(t)$ is well defined in the space of distributions and is essentially bounded with $\norma{\rho(t)}_{\Ll\infty(\reali)} \le R$. Moreover, $x \mapsto F(t,x)$ is non-decreasing, and therefore $\rho(t) \ge 0$ a.e.~in $\reali$. By Proposition~\ref{pro:convergence1} and~\eqref{eq:wass_equiv0} we easily obtain that for any $T>0$
    \begin{align*}
        &\lim_{n\to+\infty} \int_0^T d_{L,1}\left(\hat{\rho}^n(t), \rho(t)\right) {\der}t =
        \lim_{n\to+\infty} \int_0^T \int_0^L \modulo{\hat{X}^n(t,z) - X(t,z)} \,{\der}z \,{\der}t = 0,
        \\
        &\lim_{n\to+\infty} \int_0^T d_{L,1}\left(\tilde{\rho}^n(t), \rho(t)\right) {\der}t =
        \lim_{n\to+\infty} \int_0^T \int_0^L \modulo{\tilde{X}^n(t,z) - X(t,z)} \,{\der}z \,{\der}t = 0.
    \end{align*}
    Thus, $\rho$ satisfies also the last condition and $\rho(t)\in \mathcal{M}_L$.
\end{proof}

\begin{lemma}[Definition of $\check{\rho}$]\label{lem:limitcheck}
    There exists $\check{\rho}$ in $\Ll{\infty}([0,+\infty[\times [0,L];\reali)$ such that, up to a subsequence, $(\check{\rho}^n)_{n\in\N}$ converges weakly-* in $\Ll{\infty}([0,+\infty[\times  [0,L];\reali)$ to $\check{\rho}$.
\end{lemma}
\begin{proof}
    It is sufficient to observe that for any $n>0$ we have $0\le\check{\rho}^n\le R$ a.e.~on $[0,+\infty[\times[0,L]$ because, by Lemma~\ref{lem:1}, $\norma{y^n_i}_{\Ll\infty([0,+\infty[;\reali)} \le R$.
\end{proof}

We conclude this subsection by checking that the scheme is consistent with the prescribed initial condition in the limit.
\begin{proposition}\label{pro:initial}
The sequences $(\tilde{\rho}^n|_{t=0})_{n\in\N}$ and $(\hat{\rho}^n|_{t=0})_{n\in\N}$ both converge to $\bar\rho$ in the $d_{L,1}$--Wasserstein distance.
\end{proposition}
\begin{proof}
By definitions~\eqref{eq:hrn} and~\eqref{eq:trn} we have that
\begin{align*}
    &\hat{\rho}^n(0,x) = \sum_{i=0}^{N_n-1} \bar{y}^n_i \, \caratt{\left[\bar{x}_i^n, \bar{x}_{i+1}^n\right[}(x),&
    &\tilde{\rho}^n(0,x) = \ell_n \sum_{i=0}^{N_n-1} \displaystyle\dirac{\strut{\textstyle \bar{x}_i^n}}(x).
\end{align*}
Therefore $F_{\hat{\rho}^n|_{t=0}} = \hat{F}^n|_{t=0}$, $F_{\tilde{\rho}^n|_{t=0}} = \tilde{F}^n|_{t=0}$ and by~\eqref{eq:wass_equiv0}, \eqref{eq:dyi3} we have
\begin{align*}
    &d_{L,1}\left(\tilde{\rho}^n|_{t=0},\hat{\rho}^n|_{t=0}\right)
    =\norma{\tilde{F}^n|_{t=0} - \hat{F}^n|_{t=0}}_{\Ll1(\reali;\reali)}\\&
    = \sum_{i=0}^{N_n-2} \int_{\bar{x}^n_i}^{\bar{x}^n_{i+1}} \left[\vphantom{\int_{\bar{x}^n_i}^{\bar{x}^n_{i+1}}} \ell_n - \bar{y}^n_i \left[x-\bar{x}^n_i\right]\right] {\der}x
    = \ell_n \sum_{i=0}^{N_n-2} \int_{\bar{x}^n_i}^{\bar{x}^n_{i+1}} \frac{\bar{x}^n_{i+1}-x}{\bar{x}^n_{i+1}-\bar{x}^n_{i}} \,{\der}x\\&
    \le \ell_n \left[\bar{x}_{\max}-\bar{x}_{\min}\right].
\end{align*}
Hence, it is sufficient to prove that $(\tilde{\rho}^n|_{t=0})_{n\in\N}$ converges to $\bar\rho$ in the $d_{L,1}$--Wasserstein distance. By~\eqref{eq:initial_ftl} we have that
\begin{align*}
    d_{L,1}(\tilde{\rho}^n|_{t=0},\bar\rho) &
    = \norma{\tilde{F}^n|_{t=0}-F_{\bar\rho}}_{\Ll1(\reali;\reali)}
    = \sum_{i=0}^{N_n-2} \int_{\bar{x}^n_i}^{\bar{x}^n_{i+1}} \left[\ell_n \, (i+1) - \int_{-\infty}^x \bar\rho(y) \,{\der}y\right]\\ &
    = \sum_{i=0}^{N_n-2} \int_{\bar{x}^n_i}^{\bar{x}^n_{i+1}} \left[\ell_n - \int_{\bar{x}^n_i}^x \bar\rho(y) \,{\der}y\right] {\der}x
    \le \ell_n \left[\bar{x}_{\max}-\bar{x}_{\min}\right]
\end{align*}
and this concludes the proof.
\end{proof}

\subsection{$\BV$ estimates and discrete Oleinik condition}\label{sec:Oleinik}

Let us sum up what we have proven so far. The family of empirical measures $(\tilde{\rho}^n)_{n\in\N}$ converges in the scaled $1$--Wasserstein sense to a limit $\rho$ belonging to $\Ll\infty\left(\left[0,+\infty\right[; \mathcal{M}_L\right)$ and such that $0\le \rho\le R$ almost everywhere. Moreover, the empirical measure $\tilde{\rho}^n$ has a pseudo-inverse distribution function $\tilde{X}^n$ satisfying the PDE
\begin{align*}
    &\tilde{X}_t^n (t,z) = v\left(\check{\rho}^n(t,z)\right),& (t,z)\in [0,+\infty[\times [0,L],
\end{align*}
with the family $(\check{\rho}^n)_{n\in\N}$ being weakly--$*$ compact in $\Ll\infty\left([0,+\infty[\times [0,L]; \left[0,1\right]\right)$. The Wasserstein topology is a proper tool to pass to the limit the time derivative term in the above PDE, as this term is linear. But on the other hand, the weak--$*$ topology is too weak to pass to the limit $v(\check{\rho}^n)$ for a general nonlinear $v$. Moreover, there is the additional difficulty of having to check that the two limits are related in some sense.

A typical way to overcome the difficulty stated above is to provide a $\BV$ estimate for the approximating sequence $(\check{\rho}^n)_{n\in\N}$. We tackle this task in two ways. First of all, we perform a direct estimate of the total variation of $\check{\rho}^n$, and prove that such a quantity decreases in time, and is therefore uniformly bounded provided the initial datum $\bar\rho$ is $\BV$. However, this result is only partly satisfactory, as it is well known that the solution $\rho$ to~\eqref{eq:LWR_intro1} is $\BV$ even for an initial datum in $\Ll1\cap\Ll\infty$. We shall therefore prove that a uniform $\BV$ estimate of $\check{\rho}^n$ is available for an initial datum in $\Ll1\cap\Ll\infty$ provided the additional property (V3) of $v$ is prescribed. The latter task is performed by a one-sided estimate of the difference quotients of $\check{\rho}^n$, in the spirit of a discrete version of the classical Oleinik-type condition~\eqref{eq:ole_intro_hoff}, which can be considered as the main technical achievement of this paper.

We start with the following proposition.

\begin{proposition}[$\BV$ contractivity for $\BV$ initial data]\label{pro:compactness1}
Assume $v$ satisfies (V1) and (V2). If $\bar\rho $ satisfies (InBV), then for any $n\in\N$
    \begin{align*}
        &\tv\left[\hat{\rho}^n(t)\right] = \tv\left[\check{\rho}^n(t)\right] \le \tv\left[\bar\rho\right]&
        \hbox{for all }t\ge0.
    \end{align*}
\end{proposition}
\begin{proof}
    For notational simplicity, we shall omit the dependence on $t$ and $n$ whenever not necessary. By construction, see~\eqref{eq:Deltan} and~\eqref{eq:hrn}, we have that
    \begin{align*}
        &\tv\left[\hat{\rho}(0)\right] =
        \bar{y}_0 + \bar{y}_{N-1} + \sum_{i=0}^{N-2} \modulo{\bar{y}_i - \bar{y}_{i+1}}
        \\&=
        \fint_{\bar{x}_{\min}}^{\bar{x}_1} \bar\rho(y) \, {\der}y
        + \fint_{\bar{x}_{N-1}}^{\bar{x}_{\max}} \bar\rho(y) \, {\der}y
        + \sum_{i=0}^{N-2} \modulo{\fint_{\bar{x}_{i}}^{\bar{x}_{i+1}} \bar\rho(y) \, {\der}y - \fint_{\bar{x}_{i+2}}^{\bar{x}_{i+1}} \bar\rho(y) \, {\der}y}
        \le\tv\left[\bar\rho\right] .
    \end{align*}
    Moreover
    \begin{align*}
        &\frac{\der}{{\der}t} \tv\left[\hat{\rho}(t)\right] =
        \frac{\der}{{\der}t}  \left[y_0 + y_{N-1} + \sum_{i=0}^{N-2} \modulo{y_i - y_{i+1}}\right]
        =
        \dot{y}_0 + \dot{y}_{N-1}
        \\&+ \sum_{i=0}^{N-2} \sgn\left[y_i - y_{i+1}\right] \left[\dot{y}_i - \dot{y}_{i+1}\right]
        =
        \left[\vphantom{\sum_{i=1}^{N-2}} 1 + \sgn\left[y_0 - y_{1}\right]\right] \dot{y}_0
        \\&
        + \left[\vphantom{\sum_{i=1}^{N-2}} 1 - \sgn\left[y_{N-2} - y_{N-1}\right]\right] \dot{y}_{N-1}
        + \sum_{i=1}^{N-2} \left[\vphantom{\sum_{i=1}^{N-2}} \sgn\left[y_i - y_{i+1}\right] - \sgn\left[y_{i-1} - y_{i}\right]\right] \dot{y}_{i}
        .
    \end{align*}
    We claim that the latter right hand side above is $\le 0$. Indeed, assumptions (V1) and (V2) together with~\eqref{eq:dyi} imply that the following quantities are not positive
    \begin{align*}
        \left[\vphantom{\sum_{i=1}^{N-2}} 1 + \sgn\left[y_0 - y_{1}\right]\right] \dot{y}_0
        =&
        - \left[\vphantom{\sum_{i=1}^{N-2}} 1 + \sgn\left[y_0 - y_{1}\right]\right] \frac{y_0^2}{\ell} \left[v(y_1) - v(y_0)\right],
        \\
        \left[\vphantom{\sum_{i=1}^{N-2}} 1 - \sgn\left[y_{N-2} - y_{N-1}\right]\right] \dot{y}_{N-1}
        =&
        - \left[\vphantom{\sum_{i=1}^{N-2}} 1 - \sgn\left[y_{N-2} - y_{N-1}\right]\right] \times\\&\qquad\qquad\quad\quad\times \frac{y_{N-1}^2}{\ell} \left[v_{\max} - v(y_{N-1})\right],
        \\
        \left[\vphantom{\sum_{i=1}^{N-2}} \sgn\left[y_i - y_{i+1}\right] - \sgn\left[y_{i-1} - y_{i}\right]\right] \dot{y}_{i}
        =&
        - \left[\vphantom{\sum_{i=1}^{N-2}} \sgn\left[y_i - y_{i+1}\right] - \sgn\left[y_{i-1} - y_{i}\right]\right] \times\\&\qquad\qquad\quad\quad\times \frac{y_i^2}{\ell} \left[v(y_{i+1}) - v(y_i)\right] .
    \end{align*}
    Therefore, $\tv\left[\hat{\rho}(t)\right] \le \tv\left[\bar\rho\right]$ for all $t\ge0$. Finally, since $\check{\rho}^n = \hat{\rho} \circ \hat X$, $\check{\rho}^n$ is piecewise constant and it has on $i \, \ell$ the same traces as $\hat{\rho}$ on $x_i$, the statement for $\check{\rho}^n$ follows easily.
\end{proof}

We now perform our discrete Oleinik-type condition, which holds for general initial data in $\Ll1\cap\Ll\infty$.

\begin{lemma}[Discrete Oleinik-type condition]\label{lem:oleinik}
Assume $v$ satisfies (V1), (V2), and~(V3), and let $\bar\rho$ satisfy (In). Then, for any $i=0,\ldots,N_n-2$ we have
\begin{align}\label{eq:oleinik1}
    &t \, y^n_i(t) \, \left[ v\left(y^n_{i+1}(t)\right)-v\left(y^n_i(t)\right) \right] \le \,\ell_n
    &\hbox{for all }t\ge0.
\end{align}
\end{lemma}
\begin{proof}
    For notational simplicity, we shall omit the dependence on $t$ and $n$ whenever not necessary. Let
    \begin{align*}
        &z_i \doteq t \, y_i \, \left[ v\left(y_{i+1}\right)-v\left(y_i\right) \right],& i=0,\ldots,N-2,
        \\
        &z_{N-1} \doteq t \, y_{N-1} \left[v_{\max} - v(y_{N-1})\right].
    \end{align*}

    \noindent$\bullet$~\textsc{\textbf{Step~0: }$\bf z_{N-1} \le \boldsymbol\ell$.}~By~\eqref{eq:dyi1} and~(V1)
    \begin{align*}
        \dot{z}_{N-1}
        \!\!=&y_{N-1} \!\left[v_{\max} - v(y_{N-1})\right] + t \, \dot{y}_{N-1} \!\left[v_{\max} - v(y_{N-1})\right] - t \, y_{N-1} \, v'(y_{N-1}) \, \dot{y}_{N-1}\\
        =& y_{N-1}\left[v_{\max} - v(y_{N-1})\right]
        - \frac{t \, y_{N-1}^2}{\ell}\left[v_{\max} - v(y_{N-1})\right]^2
        \\&+\! \frac{t \, v'(y_{N-1}) \, y_{N-1}^3}{\ell}\left[v_{\max} \!- v(y_{N-1})\right]
        \le y_{N-1}\left[v_{\max} \!- v(y_{N-1})\right] \! \left[ 1 - \frac{z_{N-1}}{\ell}\right]\!.
    \end{align*}
    Since $z_{N-1}(0)=0$, from the above estimate we get $z_{N-1}(t) \le \ell$ for all $t\ge 0$. Indeed, assume by contradiction that there exist $t_1<t_2$ such that $z_{N-1}(t_1)=\ell$ and $z_{N-1}(t)>\ell$ for all $t\in\,\, ]t_1,t_2[$. The above estimate implies for $t\in\,\,]t_1,t_2[$
    \begin{align*}
    z_{N-1}(t) & =z_{N-1}(t_1) + \int_{t_1}^t y_{N-1}(s)\left[v_{\max} \!- v(y_{N-1}(s))\right] \! \left[ 1 - \frac{z_{N-1}(s)}{\ell}\right]\! {\der} s\\
    & \leq z_{N-1}(t_1) = \ell,\
    \end{align*}
    which gives a contradiction.

    \noindent$\bullet$~\textsc{\textbf{Step~1: }$\bf z_{i+1} \le \boldsymbol\ell \Rightarrow z_{i} \le \boldsymbol\ell$.}~Let $i\in \{0,\ldots,N-3\}$ and assume $z_{i+1} \le \ell$. From~\eqref{eq:dyi2} and~(V1) we get
    \begin{align*}
        \dot{z}_i =&\, y_i\, \left[ v\left(y_{i+1}\right)-v\left(y_i\right) \right] + t \, \dot{y}_i \left[ v\left(y_{i+1}\right)-v\left(y_i\right) \right]  + t \, y_i \left[ v'(y_{i+1}) \, \dot{y}_{i+1} - v'(y_{i}) \, \dot{y}_{i}\right]\\
        =&\, y_i\, \left[ v\left(y_{i+1}\right)-v\left(y_i\right) \right] -\frac{t \, y_i^2}{\ell}\left[v(y_{i+1})-v(y_i)\right]^2 \\
        & \, + t \, y_i \left[ - \frac{v'(y_{i+1}) \, y_{i+1}^2}{\ell}\left[v(y_{i+2})-v(y_{i+1})\right] + \frac{v'(y_{i}) \, y_{i}^2}{\ell}\left[v(y_{i+1})-v(y_{i})\right]\right]\\
        =&\, y_i \!\left[ v\!\left(y_{i+1}\right)-v\!\left(y_i\right) \right]\!
        -\! \frac{y_i}{\ell}\!\left[v\!\left(y_{i+1}\right)-v\!\left(y_i\right)\right] z_i
        \!- \!\frac{v'\!(y_{i+1}) \, y_i \, y_{i+1}}{\ell} \, z_{i+1}
        \!+ \!\frac{v'\!(y_{i}) \, y_{i}^2}{\ell} \, z_{i} .
    \end{align*}
    Since $\sgn_+\left[z_i\right] = \sgn_+\left[ v\left(y_{i+1}\right)-v\left(y_i\right) \right] = \sgn_+\left[ y_i - y_{i+1} \right]$ for all $t>0$, from the assumption on $z_{i+1}$ we easily obtain
    \begin{align*}
        \frac{\der}{{\der}t}[z_i]_+ =&\, y_i\, \left[ v\left(y_{i+1}\right)-v\left(y_i\right) \right]_+ -\frac{y_i}{\ell}\left[v(y_{i+1})-v(y_i)\right]_+ [z_i]_+
        \\
        &\, - \frac{v'(y_{i+1}) \, y_i \, y_{i+1}}{\ell} \, \sgn_+[z_i] \, z_{i+1} +  \frac{v'(y_{i}) \, y_{i}^2 }{\ell}\, [z_i]_+ \\
        \le& \, y_i \!\left[ v\!\left(y_{i+1}\right)\!-\!v\!\left(y_i\right) \right]_+ \! \left[ 1\!-\! \frac{[z_i]_+}{\ell}\right]\!
        \!-\! v'\!(y_{i+1}) \, y_i \, y_{i+1} \, \sgn_+[z_i] \!+\! \frac{v'\!(y_{i}) \, y_{i}^2}{\ell} \, [z_i]_+ .
    \end{align*}
    Condition~(V3) prescribes that the function $y\mapsto y \, v'(y)$ is non-increasing, which gives
    \begin{align*}
      \frac{\der}{{\der}t}[z_i]_+ & \le y_i\, \left[ v\left(y_{i+1}\right)-v\left(y_i\right) \right]_+ \left[ 1- \frac{[z_i]_+}{\ell}\right]
      - v'(y_i) \, y_i^2 \, \sgn_+[z_i] + \frac{v'(y_{i}) \, y_{i}^2}{\ell} \, [z_i]_+\\
        & \ =  y_i \left[\vphantom{\frac{[z_i]_+}{\ell}} \left[ v\left(y_{i+1}\right)-v\left(y_i\right) \right]_+ - v'(y_i) \, y_i \right] \left[ 1- \frac{[z_i]_+}{\ell}\right].
    \end{align*}
    Now, as $v'\le 0$, and since $z_i(0)=0$, by a similar comparison argument as the one at the end of Step 1 we get that $z_i(t)_+\le \ell$ for all $t\ge 0$.

    \noindent$\bullet$~\textsc{\textbf{Step~2: }$\bf z_{N-2} \le \boldsymbol\ell$.}~From analogous computations as in previous step, we get
    \begin{align*}
      \frac{\der}{{\der}t}[z_{N-2}]_+ =& \, y_{N-2}\, \left[ v\!\left(y_{N-1}\right)-v\!\left(y_{N-2}\right) \right]_+\!\! -\frac{y_{N-2}}{\ell}\left[v\!\left(y_{N-1}\right) - v\!\left(y_{N-2}\right)\right]_+ [z_{N-2}]_+ \\
        & \, - \frac{v'\!(y_{N-1}) \, y_{N-2} \, y_{N-1}}{\ell} \, \sgn_+[z_{N-2}] \, z_{N-1} \!+\! \frac{v'\!(y_{N-2}) \, y_{N-2}^2}{\ell} \, [z_{N-2}]_+,
    \end{align*}
    and we can use the monotonicity of $y\mapsto y \, v'(y)$ and Step~0 to get
    \begin{align*}
      \frac{\der}{{\der}t}[z_{N-2}]_+ \le& \, y_{N-2}\, \left[ v\left(y_{N-1}\right)-v\left(y_{N-2}\right) \right]_+ \left[1-\frac{[z_{N-2}]_+}{\ell}\right] \\
      & \, -  v'(y_{N-2}) \, y_{N-2}^2 \, \sgn_+[z_{N-2}] + \frac{v'(y_{N-2}) \, y_{N-2}^2}{\ell} \, [z_{N-2}]_+
      \\
      =& \, y_{N-2}\, \left[\vphantom{\frac{[z_{N-2}]_+}{\ell}} \left[ v\left(y_{N-1}\right)-v\left(y_{N-2}\right) \right]_+ -  v'(y_{N-2}) \, y_{N-2}\right]
      \left[1-\frac{[z_{N-2}]_+}{\ell}\right] .
    \end{align*}
    Again, $v'\le 0$ and $z_{N-2}(0)=0$ imply that $z_{N-2}(t)_+ \le \ell$ for all $t\ge 0$.

    \noindent$\bullet$~\textbf{\textsc{Conclusion}}. The estimate~\eqref{eq:oleinik1} is proven recursively: Step~2 provides the first step with $i=N-2$, whereas Step~1 proves that the estimate holds for all $i\in \{0,\ldots,N-3\}$.
\end{proof}

\begin{corollary}\label{cor:oleinik_discrete_2}
    Assume $v$ satisfies (V1),~(V2),~and~(V3), and let $\bar\rho$ satisfy (In). Then, for any $i\in \{0,\ldots,N-2\}$ we have
    \begin{align}\label{eq:oleinik_discrete_2}
        & v\left(\hat{\rho}^n\left(t,x_i^n(t)\right)\right) - v\left(\hat{\rho}^n\left(t,x_{i+1}^n(t)\right)\right)
        \le
        \frac{x_{i+1}^n(t) -x_i^n(t)}{t}&
        \hbox{for all }t >0 .
    \end{align}
\end{corollary}
\begin{proof}
    The statement follows from Lemma~\ref{lem:oleinik} and the definitions of $\hat\rho^N$ and $y_i$.
\end{proof}

In the following proposition we prove uniform bounds on the total variation of $v\left(\check{\rho}^n\right)$ and $v\left(\hat{\rho}^n\right)$. Let us emphasize that the regularising effect $\Ll\infty \mapsto \BV$ implies that the $\BV$ estimate eventually \emph{blows up} as $t\searrow 0$.

\begin{proposition}[Uniform $\BV$ estimates for $v\left(\check{\rho}^n\right)$ and $v\left(\hat{\rho}^n\right)$]\label{pro:compactness}
    Assume $v$ satisfies the properties (V1),~(V2), and~(V3), and let $\bar\rho$ satisfy (In). Let $\delta>0$. Then
    \begin{enumerate}[(i)]
      \item $(v\left(\hat{\rho}^n\right))_{n\in\N}$ is uniformly bounded in $\Ll\infty\left(\left[\delta,+\infty\right[;\,\BV(\reali;[v(R),v_{\max}])\right)$;
      \item $(v\left(\check{\rho}^n\right))_{n\in\N}$ is uniformly bounded in $\Ll\infty\left(\left[\delta,+\infty\right[;\,\BV([0,L];[v(R),v_{\max}])\right)$.
    \end{enumerate}
    More precisely, for any $n\in\N$
    \begin{align*}
        &\tv\left[v(\hat{\rho}^n(t))\right] = \tv\left[v(\check{\rho}^n(t))\right] \le
        C_\delta&
        \hbox{for all }t\ge\delta,
    \end{align*}
    where $C_\delta \doteq \left[3 \, (v_{\max}-v(R))+ 2 \, \frac{\bar{x}_{\max} -\bar{x}_{\min}}{\delta}\right]$.
\end{proposition}

\begin{proof}
    For notational simplicity, we shall omit the dependence on $n$. We set
    \begin{align*}
      &\hat\sigma(t,x) \doteq v(\hat{\rho}(t,x)) +\frac{1}{t}\sum_{i=0}^{N-1}x_i(t) \, \caratt{\left[x_i(t),x_{i+1}(t)\right[}(x)&
      \hbox{for all }x\in \reali.
    \end{align*}
    We claim that, for any fixed $t\ge0$, the map $x \mapsto \hat\sigma(t,x)$ is a piecewise constant, non-decreasing function on $\left[x_0(t),x_{N}(t)\right[$. To see this, we first notice that the map $x \mapsto \hat\sigma(t,x)$ is constant on the interval $\left[x_i(t),x_{i+1}(t)\right[$, $i=0,\ldots,N-1$. On the other hand, $\hat\sigma(t)$ is non-decreasing on the potential discontinuity points $x_i(t)$, $i=1,\ldots,N-1$, in view of~\eqref{eq:oleinik_discrete_2}. Now, from~\eqref{eq:ftl} and the discrete maximum principle in Lemma~\ref{lem:1} we know that for any $x\in\left[x_0(t),x_{N}(t)\right[$
    \begin{equation*}
        \frac{\bar{x}_{\min}}{t}+2v(R)
        \le
        \hat\sigma(t,x)
        \le
        v_{\max} + \frac{1}{t}  \left[\bar{x}_{\max}+v_{\max}\,t\right]
        =
        2v_{\max} + \frac{\bar{x}_{\max}}{t} .
    \end{equation*}
    Hence $\hat\sigma$ is uniformly bounded in $\Ll\infty\left(\left[\delta,+\infty\right[;\BV(\reali;\reali)\right)$ with
    \begin{align*}
        &\sup_{t \ge \delta}\tv\left[\hat\sigma(t)\right] \le
        \left[2(v_{\max}-v(R)) + \frac{\bar{x}_{\max} -\bar{x}_{\min}}{\delta}\right] .
    \end{align*}
    Therefore, also $v(\hat{\rho})$ is uniformly bounded in $\Ll\infty\left(\left[\delta,+\infty\right[;\BV(\reali;\reali)\right)$ because by triangular inequality
    \begin{align*}
        \tv\left[v(\hat{\rho}(t))\right]
        &\le
        \tv\left[\hat\sigma(t)\right] + \tv\left[ \frac{1}{t}\sum_{i=0}^{N-1}x_i(t) \, \caratt{\left[x_i(t),x_{i+1}(t)\right[}\right]
        \\
        &=
        \tv\left[\hat\sigma(t)\right] + \frac{1}{t} \left[\bar{x}_{\max}-\bar{x}_{\min}+(v_{\max}-v(R))\,t\right]
        \le
        C_\delta ,
    \end{align*}
    for $t\ge \delta$.
    Finally, since $\check{\rho}^n = \hat{\rho} \circ \hat X$, $\check{\rho}^n$ is piecewise constant and on $i \, \ell$ has the same traces as $\hat{\rho}$ on $x_i$, the statement for $\check{\rho}^n$ follows easily.
\end{proof}

\subsection{Time continuity and compactness.}\label{sec:compactness}

\begin{proposition}[Uniform $\Ll1$--continuity in time of $\check{\rho}^n$]\label{pro:continuity}
Under the assumptions of Proposition~\ref{pro:compactness}, for any $\delta>0$ we have
\begin{align*}
    &\int_0^L \modulo{\check{\rho}^n(t,z) -\check{\rho}^n(s,z)} \,{\der}z \le R^2 \left[ C_\delta + \left(v_{\max}-v(R)\right)\right]  \modulo{t-s}&
    \hbox{for all }t,s \ge \delta ,
\end{align*}
with $C_\delta$ defined in Proposition~\ref{pro:compactness}.
\end{proposition}

\begin{proof}
By~\eqref{eq:dyi}, we compute for $t>s>\delta$,
\begin{align*}
    &\int_0^L \modulo{\check{\rho}^n(t,z) -\check{\rho}^n(s,z)} \,{\der}z
    =\sum_{i=0}^{N_n-1} \ell_n \, \modulo{y^n_i(t) - y^n_i(s)}
    = \sum_{i=0}^{N_n-1} \ell_n \, \modulo{\int_s^t \dot{y}^n_i(\tau) \,{\der}\tau}
    \\
    &=\!\! \sum_{i=0}^{N_n\!-2} \modulo{\int_s^t\!\!\! y^n_i\!(\tau)^2 \!\left[v\!\left(y^n_{i+1}\!(\tau)\right) - v\!\left(y^n_i\!(\tau)\right)\right] {\der}\tau}
    +\! \int_s^t\!\!\! y^n_{N_n-1}\!(\tau)^2 \!\left[v_{\max} \!- v\!\left(y^n_{N_n-1}\!(\tau)\right)\right] {\der}\tau .
\end{align*}
Therefore, by Lemma~\ref{lem:1}
\begin{align*}
    &\int_0^L \modulo{\check{\rho}^n(t,z) -\check{\rho}^n(s,z)} \,{\der}z
    \\
    &\le\!
     \int_s^t\! \left[\sum_{i=0}^{N_n-2} y^n_i(\tau)^2 \, \modulo{v\!\left(y^n_{i+1}(\tau)\right) - v\!\left(y^n_i(\tau)\right)}
    + y^n_{N_n-1}(\tau)^2 \left[v_{\max} - v\!\left(y^n_{N_n-1}(\tau)\right)\right] \right] \!{\der}\tau
    \\
    &\le\!
     R^2\int_s^t\! \left[\sum_{i=0}^{N_n-2} \modulo{v\!\left(y^n_{i+1}(\tau)\right) - v\!\left(y^n_i(\tau)\right)}
    + \left( v_{\max} - v(R) \right) \right] \!{\der}\tau\\
    &\le\!
    R^2 \int_s^t\! \left[\vphantom{\sum_{i=0}^{N_n-2}}\tv\left[v\!\left(\check{\rho}^n(\tau)\right)\right] + \left( v_{\max} - v(R)\right) \right] \!{\der}\tau .
\end{align*}
Then it is sufficient to apply the estimate in Proposition~\ref{pro:compactness} to complete the proof.
\end{proof}

\begin{proposition}[Uniform Wasserstein time continuity of $\hat\rho^n$]\label{pro:wass_continuity}
    Assume $v$ satisfies the properties (V1) and (V2), and let $\bar \rho$ satisfy the assumption (In). For any $n \in \N$ we have
    \begin{align*}
        &d_{L,1}\left(\hat{\rho}^n(t),\hat{\rho}^n(s)\right) \!\le 2 L \max\{\modulo{v_{\max}},\modulo{v(R)},[v_{\max}-v(R)]\}\, \modulo{t-s}&
        &\hbox{for all }s,t \ge0.
    \end{align*}
\end{proposition}

\begin{proof}
    By~\eqref{eq:wass_equiv0}, \eqref{eq:Deltan} and~\eqref{eq:ftl}, we compute for any $t>s\ge0$
    \begin{align*}
        & d_{L,1} \left(\hat{\rho}^n(t),\hat{\rho}^n(s)\right) =
        \norma{\hat X^n(t) - \hat X^n(s)}_{\Ll1([0,L];\reali)} \\
        &= \sum_{i=0}^{N_n-1} \int_{i\,\ell_n}^{(i+1)\,\ell_n} \modulo{\hat X^n(t,z) - \hat X^n(s,z)} {\der}z \\
        &= \sum_{i=0}^{N_n-1} \int_{i\,\ell_n}^{(i+1)\,\ell_n} \modulo{x^n_i(t) + \frac{z-i\,\ell_n}{y^n_i(t)} - x^n_i(s) - \frac{z-i\,\ell_n}{y^n_i(s)}} {\der}z \\
        &\le \sum_{i=0}^{N_n-1}\ell_n \modulo{x^n_i(t) -x^n_i(s)}
        +\sum_{i=0}^{N_n-1} \modulo{y^n_i(t)^{-1} - y^n_i(s)^{-1}} \int_{i\,\ell_n}^{(i+1)\,\ell_n} (z-i\,\ell_n) \,{\der}z \\
        &\le \sum_{i=0}^{N_n-1}\ell_n \int_s^t \modulo{v\left(y^n_i(\tau)\right)} {\der}\tau + \sum_{i=0}^{N_n-1} \frac{\ell_n^2}{2} \int_s^t \modulo{\frac{\der}{{\der}\tau} \left[y^n_i(\tau)^{-1}\right]} {\der}\tau\\
        &\le L\,\max\{\modulo{v_{\max}}\,,\modulo{v(R)}\}\,(t-s)\\
        & +  \frac{\ell_n}{2} \!\int_s^t\! \left[\sum_{i=0}^{N_n-2} \!\modulo{ v\!\left(y^n_{i+1}(\tau)\right) - v\!\left(y^n_i(\tau)\right)} + [v_{\max} \!- v\!\left(y^n_{N_n-1}(\tau)\right)]\right] {\der}\tau \\
        &= L\,\max\{\modulo{v_{\max}}\,,\modulo{v(R)}\}\,(t-s)
        +  L\,\int_s^t \left[v_{\max} - v(R)\right] {\der}\tau \\
        &\le 2 \, L \, \max\{\modulo{v_{\max}}\,,\modulo{v(R)}\,,[v_{\max}-v(R)]\} \, (t-s)
    \end{align*}
    and this concludes the proof.
\end{proof}

We now recall a generalization of Aubin-Lions lemma, which uses the Wasserstein distance as a replacement of a negative Sobolev norm, proven in~\cite[Theorem 2]{rossisavare}, which we present here in a version adapted to our case. In order to have the paper self-contained, we first recall the precise statement
of~\cite[Theorem 2]{rossisavare} (see also the adapted version in~\cite{DFMatthes}).

\begin{theorem}[Theorem~2 from~\cite{rossisavare}]\label{thm:rossi}
    On a separable Banach space $\mathbb{X}$, let be given
    \begin{enumerate}
        \item[(F)] a \emph{normal coercive integrand} $\mathfrak{F} \,\colon\, \mathbb{X}\to[0,+\infty]$, i.e.,
        $\mathfrak{F}$ is lower semi-continuous and its sublevels are relatively compact in $\mathbb{X}$;
        \item[(g)] a \emph{pseudo-distance} $\mathfrak{g} \,\colon\, \mathbb{X}\times \mathbb{X}\to[0,+\infty]$, i.e.,
        $\mathfrak{g}$ is lower semi-continuous,
        and if $\nu,\mu\in \mathbb{X}$ are such that $\mathfrak{g}(\nu,\mu)=0$, $\mathfrak{F}[\nu]<+\infty$ and $\mathfrak{F}[\mu]<+\infty$, then $\nu=\mu$.
    \end{enumerate}
    Let further $U$ be a set of measurable functions $\nu \,\colon\, \left]0,T\right[\to \mathbb{X}$, with a fixed $T>0$.
    Under the hypotheses that
    \begin{align}
        \label{eq:savare_hypo}
        &\sup_{\nu\in U}\int_0^T \mathfrak{F}\left[\nu(t)\right] {\der}t<+\infty&
        &\text{and}&
        \lim_{h\downarrow0} \left[\sup_{\nu\in U}\int_0^{T-h} \mathfrak{g}\left(\nu(t+h),\nu(t)\right) {\der}t\right]=0,
    \end{align}
    Then $U$ is strongly relatively compact in $\Ll1(\left]0,T\right[; \mathbb{X})$.
\end{theorem}

\begin{theorem}[Generalized Aubin-Lions lemma]\label{pro:aubin}
    Let $T,\,L > 0$ and $I \subset \reali$ be a bounded open convex interval. Assume $w \,\colon\, \reali \to \reali$ is a Lipschitz continuous and strictly monotone function.  Let $(\rho^n)_{n\in\N}$ be a sequence in $\Ll\infty\left(\left]0,T\right[\times\reali\right)\cap \mathcal{M}_L$ such that
    \begin{enumerate}
      \item $\rho^n \,\colon\, \left]0,T\right[\to \Ll1\left(\reali\right)$ is measurable for all $n \in \N$;
      \item $\spt\left(\rho^n(t)\right) \subseteq I$ for all $t \in \left]0,T\right[$ and $n \in \N$;
      \item $\sup_{n \in \N}\int_{0}^{T} \left[\vphantom{\int_{0}^{T}} \norma{w\left(\rho^n(t)\right)}_{\Ll1(I)}+\tv\left[w\left(\rho^n(t)\right)\right] \right] {\der}t <+\infty$;
      \item there exists a constant $C$ depending only on $T$ such that $d_{L,1}\left(\rho^n(s),\rho^n(t)\right)\le C \, \modulo{t-s}$ for all $s, t \in \left]0,T\right[$ and $n\in\N$.
    \end{enumerate}
    Then, $(\rho^n)_{n\in\N}$ is strongly relatively compact in $\Ll1(\left]0,T\right[\times \reali)$.
\end{theorem}

\begin{proof}
    We want to use Theorem~\ref{thm:rossi} with
    \begin{align*}
        &\mathbb{X}\doteq\Ll1\left(I\right),&
        &U\doteq(\rho^n)_{n\in\N},\\
        &\mathfrak{F}[\nu] \doteq \norma{w(\nu)}_{\Ll1(I)} + \tv\left[w(\nu)\right],&
        &\mathfrak{g}(\nu,\mu) \doteq \begin{cases}
        d_{L,1}(\nu,\mu)& \hbox{if}\quad \nu,\mu\in \mathcal{M}_L,\\
        +\infty & \hbox{otherwise}.
        \end{cases}
    \end{align*}
We first have to prove that $\mathfrak{F}$, $\mathfrak{g}$ and $U$ satisfy the corresponding hypotheses in Theorem~\ref{thm:rossi}.

\noindent\textbf{(F)}~Assume that $(\nu^n)_{n\in\N}$ converges to $\nu$ strongly in $\Ll1\left(I\right)$. Since $w$ is Lipschitz continuous, $\left(w(\nu^n)\right)_{n\in\N}$ converges to $w(\nu)$ strongly in $\Ll1\left(I\right)$. Hence, for the lower semi-continuity of the total variation w.r.t.~the $\Ll1$--norm, see~\cite[Theorem~1 on page~172]{evans1991measure}, we have that $\tv\left[w(\nu)\right] \le \liminf_{n\to+\infty}\tv\left[w(\nu^n)\right]$. Thus $\mathfrak{F}\left[\nu\right] \le \liminf_{n\to+\infty} \mathfrak{F}\left[\nu^n\right]$ and $\mathfrak{F}$ is l.s.c.~in $\mathbb{X}$. Finally, consider a sequence $(\nu^n)_{n\in\N}$ belonging to a sublevel of $\mathfrak{F}$, namely $\sup_{n \in \N} \mathfrak{F}\left[\nu^n\right] < + \infty$. For the compactness of $\BV$ in $\Ll1$ on bounded open convex intervals and for basic properties of the $\Ll1$--convergence, see~\cite[Theorem~4 on page~176 and Theorem~5 on page~21]{evans1991measure}, up to a subsequence $\left(w(\nu^n)\right)_{n\in\N}$ converges to $\bar{w}$ in $\Ll1$ and a.e.~on $I$. Since $w$ is continuous and strictly monotone, $(\nu^n)_{n\in\N}$ is uniformly bounded in $\Ll\infty$ (consequence of the uniform bound on the total variation) and converges to $\bar\nu \doteq w^{-1}(\bar{w})$ a.e.~on $I$ and therefore, by the Lebesgue dominated convergence theorem, the convergence is also in $\Ll1$.

\noindent\textbf{(g)}~Proceeding as before and applying lower semi-continuity of the $1$--Wasserstein distance w.r.t.~the $\Ll1$--norm give that $\mathfrak{g}$ is l.s.c.~in $\mathbb{X}\times\mathbb{X}$. Finally, if $\mathfrak{F}[\nu]<+\infty$, $\mathfrak{F}[\mu]<+\infty$ and $\mathfrak{g}(\mu,\nu)=0$, then $w(\mu), w(\nu)$ are in $\BV$, $\nu,\mu\in \mathcal{M}_L$, and $d_{L,1}\left(\mu,\nu\right) = 0$. Hence we have $\mu=\nu$.

\noindent\textbf{(U)}~Conditions in~\eqref{eq:savare_hypo} follow directly from the hypotheses~(3) and~(4).

\noindent Hence we can apply Theorem~\ref{thm:rossi} and obtain the strong compactness in $\Ll1(\left]0,T\right[\times I)$. Finally, recalling the hypothesis~2 concludes the proof.
\end{proof}

\subsection{Convergence to entropy solutions.}\label{sec:convergence2}

In the next proposition we collect the previous compactness results to get strong convergence.
\begin{proposition}\label{pro:sub_sequence}
    Let $\check{\rho}$ be defined as in Lemma~\ref{lem:limitcheck} and $\rho$ as in Proposition~\ref{pro:convergence3}. Under the assumptions in Theorem~\ref{thm:main} we have that
    \begin{enumerate}[(i)]
      \item the sequence $(\check{\rho}^n)_{n\in\N}$ converges up to a subsequence to $\check{\rho}$ almost everywhere and strongly in $\Llloc1$ on $\left]0,+\infty\right[\times[0,L]$;
      \item the sequence $(\hat{\rho}^n)_{n\in\N}$ converges up to a subsequence to $\rho$ almost everywhere and strongly in $\Llloc1$ on $\left]0,+\infty\right[\times\reali$;
      \item if $\bar\rho$ satisfies also (InBV), then the sequence $(\check{\rho}^n)_{n\in\N}$ converges up to a subsequence to $\check{\rho}$ strongly in $\Llloc1$ on $\left[0,+\infty\right[\times[0,L]$.
    \end{enumerate}
\end{proposition}

\begin{proof}
We already know from Proposition~\ref{pro:convergence3} that both $(\hat{\rho}^n)_{n\in\N}$ and $(\tilde{\rho}^n)_{n\in\N}$, defined respectively by~\eqref{eq:hrn} and~\eqref{eq:trn}, converge in the topology of $\Llloc1\left(\left[0,+\infty\right[; d_{L,1}\right)$ to the density $\rho\in \Ll\infty\left(\left[0,+\infty\right[; \mathcal{M}_L\right)$ with $0\le \rho\le R$. From Proposition~\ref{pro:convergence1} we know that both $(\hat{X}^n)_{n\in\N}$ and $(\tilde{X}^n)_{n\in\N}$, defined respectively by~\eqref{eq:hXn} and~\eqref{eq:tXn}, converge strongly in $\Llloc1\left( \left[0,+\infty\right[ \times \left[0 ,L\right]\right)$ to $X \in \Ll\infty\left(\left[0,+\infty\right[ \times \left[0,L\right]\right)$, the pseudo-inverse of $F$, the cumulative distribution of $\rho$. Finally, from Lemma~\ref{lem:limitcheck} we know that $(\check{\rho}^n)_{n\in\N}$, defined by~\eqref{eq:crn}, converges up to a subsequence to $\check{\rho}$ weakly-* in $\Ll{\infty}([0,+\infty[\times  [0,L])$.

\noindent$\bullet$~\textbf{\textsc{Step~1.}} \emph{Strong convergence of $(\check\rho^n)_{n\in\N}$ for general initial datum in $\mathcal{M}_L \cap \Ll\infty$.}

\noindent Let $\bar\rho$ satisfy (In). For any fixed $\delta>0$, we know from Proposition~\ref{pro:compactness} that $\left(v(\check{\rho}^n)\right)_{n\in\N}$ is uniformly bounded in $\Ll\infty\!\left(\left[\delta,+\infty\right[;\,\BV([0,L]; [v(R),v_{\max}])\right)$.
Furthermore, from Proposition~\ref{pro:continuity} we easily obtain that
\begin{align*}
    &\int_0^L \modulo{v(\check{\rho}^n(t,z)) -v(\check{\rho}^n(s,z))} \,{\der}z \le \lip\left(v\right) R^2 \left[ C_\delta + \left(v_{\max}-v(R)\right)\right]  \modulo{t-s}\,,
\end{align*}
for all $t,s\ge\delta$.
Therefore, we can apply Helly's theorem in the form~\cite[Theorem~2.4]{BressanBook} to get that $(v(\check{\rho}^n))_{n\in\N}$ is strongly compact in $\Llloc1(\left[\delta,+\infty\right[ \times[0,L])$. Hence, by the monotonicity of $v$, up to a subsequence $(\check{\rho}^n)_{n\in\N}$ converges strongly in $\Llloc1$ and a.e.~on $\left[\delta,+\infty\right[ \times [0,L]$ to $\check{\rho}$. Finally, since $\delta>0$ is arbitrary, the proof of~$(i)$ is complete.

\noindent$\bullet$~\textbf{\textsc{Step~2.}} \emph{Strong convergence of $(\hat\rho^n)_{n\in\N}$ for general initial datum in $\mathcal{M}_L \cap \Ll\infty$.}

\noindent Let $\bar\rho$ satisfy (In) and fix $T,\delta>0$ with $\delta<T$. We want to prove that $(\rho^n)_{n\in\N}$ with $\rho^n(t,x) \doteq \hat\rho^n(t+\delta,x)$ satisfies the hypotheses of Theorem~\ref{pro:aubin} with
\[I = \left]\bar{x}_{\min}+v(R)\,(T+\delta)-1, \bar{x}_{\max} + v_{\max}\,(T+\delta)+1\right[\]
and $w=v$. The hypotheses~1 and~2 are satisfied because by~\eqref{eq:hrn} we have that $\norma{\hat\rho^n(t)}_{\Ll1(\reali)} = L$ and $\spt\left(\hat\rho^n(t)\right) \subset I$ for all $t\in \left[0,T+\delta\right]$. By Proposition~\ref{pro:compactness}, the hypothesis~3 holds true because
\begin{align*}
    \int_{\delta}^{T+\delta} \left[\vphantom{\int_{0}^{T}} \norma{v\left(\hat\rho^n(t)\right)}_{\Ll1(I;\reali)}+\tv\left[v\left(\hat\rho^n(t)\right)\right] \right] {\der}t
    \le \left[\max\{\modulo{v_{\max}},\modulo{v(R)}\} \, \modulo{I} + C_\delta\right] T.
\end{align*}
Finally, the hypothesis~4 follows directly from Proposition~\ref{pro:wass_continuity}. Hence, we can apply Theorem~\ref{pro:aubin} to obtain that $(\hat{\rho}^n)_{n\in\N}$ is strongly compact in $\Ll1(\left]\delta,T\right[\times \reali;\reali)$. By the uniqueness of the limit in the $\Ll1(\left]\delta,T\right[;\; d_{L,1})$ topology, up to a subsequence $(\hat{\rho}^n)_{n\in\N}$ converges strongly in $\Ll1$ and a.e.~on $\left]\delta,T\right[ \times \reali$ to $\rho$. Finally, since $T>\delta>0$ are arbitrary, the proof of~$(ii)$ is complete.

\noindent$\bullet$~\textbf{\textsc{Step~3.}} \emph{Strong convergence for initial datum in $\BV$.}

\noindent Let $\bar\rho$ satisfy (InBV). The result in Proposition~\ref{pro:compactness1} ensures that both $(\hat{\rho}^n)_{n\in\N}$ and $(\check{\rho}^n)_{n\in\N}$ are uniformly bounded in $\Ll\infty([0,+\infty[;\,\BV(\reali))$. Hence, we can repeat the proof of Proposition~\ref{pro:continuity} (we omit the details) to obtain that
\begin{align*}
    &\int_0^L \modulo{\check{\rho}^n(t,z) - \check{\rho}^n(s,z)} \,{\der}z \le \left[ \lip\left(v|_{[0,R]}\right) \tv(\bar\rho) + (v_{\max}-v(R))\right]  \modulo{t-s}\,,
\end{align*}
for all $t,s\ge 0$.
Therefore, Helly's theorem implies the desired compactness. Moreover, we can use Theorem~\ref{pro:aubin} with $w$ being the identity function on $[0,R]$, and obtain the desired compactness of $\hat\rho^n$.
\end{proof}

We now prove that the two limits $\check\rho$ and $\rho$ are related.

\begin{proposition}\label{pro:support}
Let $F$ be the cumulative distribution of $\rho$ as defined in Proposition~\ref{pro:convergence2}. Then
\begin{align*}
    &\check{\rho}\left(t, F(t,x)\right) = \rho(t,x) & \hbox{for a.e.~}(t,x)\hbox{ in }\spt (\rho).
\end{align*}
\end{proposition}

\begin{proof}
By definition~\eqref{eq:crn} and Lemma~\ref{lem:change_of_variable}, for any $\varphi \in \Cc\infty([0,T]\times \reali)$ we have
\begin{align*}
    \int_0^T\!\! \int_0^L \check{\rho}^n (t,z) \, \varphi\left(t,\hat X^n(t,z)\right)  {\der}z\,{\der}t
    &=\!
    \int_0^T\!\! \int_0^L \hat{\rho}^n \left(t,\hat X^n(t,z)\right) \, \varphi\left(t,\hat X^n(t,z)\right)  {\der}z\,{\der}t\\
    &=\!
    \int_0^T\!\! \int_\reali \hat{\rho}^n(t,x)^2 \,\varphi(t,x) \,{\der}x\,{\der}t.
\end{align*}
By extracting the a.e.~convergent subsequence provided in Proposition~\ref{pro:sub_sequence} (and by extracting, if necessary, a further subsequence), we can send $n\rightarrow +\infty$ in the above identity and use the Lebesgue dominated convergence theorem (as the support of $\check{\rho}^n$ and $\hat{\rho}^n$ are uniformly bounded w.r.t.~$n$) to get
\begin{align*}
    & \int_0^T \int_0^L \check{\rho} (t,z) \,\varphi\left(t, X(t,z)\right)  {\der}z\,{\der}t = \int_0^T \int_\reali \rho(t,x)^2 \,\varphi(t,x) \,{\der}x\,{\der}t.
\end{align*}
By changing variable $z=F(t,x)$ in the first integral above, we get
\begin{align}\label{eq:equiv_1}
   & \int_0^T \int_\reali \check{\rho} \left(t,F(t,x)\right)  \rho(t,x) \, \varphi(t, x) \,{\der}x\,{\der}t =
   \int_0^T \int_\reali \rho(t,x)^2 \, \varphi(t,x) \,{\der}x\,{\der}t,
\end{align}
and this concludes the proof.
\end{proof}

In the next proposition we prove that $\rho$ is a weak solution in the sense of~\eqref{eq:weaksol}.

\begin{proposition}
The limit function $\rho$ defined in Proposition~\ref{pro:convergence3} is a weak solution in the sense of~\eqref{eq:weaksol}.
\end{proposition}

\begin{proof}
Let $\varphi \in \Cc\infty\left(\left[0,+\infty\right[ \times \reali; \reali\right)$. By~\eqref{eq:ftl}, \eqref{eq:tXn} and~\eqref{eq:crn}, for all $n$ we have
\begin{align*}
    &\int_{\reali_+}\int_0^L \left[\vphantom{\frac{\der}{{\der}t}} v\left(\check{\rho}^n(t,z)\right)  \varphi_x \left(t,\tilde{X}^n(t,z)\right)\right]  {\der}z\,{\der}t
    \\&=
    \sum_{i=0}^{N_n-1} \int_{\reali_+} \int_{i\,\ell_n}^{(i+1)\,\ell_n} \left[\vphantom{\frac{\der}{{\der}t}} v\left(y^n_i(t)\right) \varphi_x \left(t,x_i^n(t)\right) \right] {\der}z\,{\der}t
    \\
    &=
    \sum_{i=0}^{N_n-1} \int_{\reali_+} \int_{i\,\ell_n}^{(i+1)\,\ell_n} \left[\vphantom{\frac{\der}{{\der}t}} \dot{x}_i^n(t) \, \varphi_x\left(t,x_i^n(t)\right) \right] {\der}z\,{\der}t
    \\&=
    \sum_{i=0}^{N_n-1} \int_{\reali_+} \int_{i\,\ell_n}^{(i+1)\,\ell_n} \left[ \frac{\der}{{\der}t} \varphi\left(t,x_i^n(t)\right) - \varphi_t\left(t,x_i^n(t)\right) \right] {\der}t
    \\
    &=- \int_0^L \varphi\left(0,\tilde{X}^n(0,z)\right) {\der}z
    - \int_{\reali_+}\int_{0}^{L} \varphi_t\left(t,\tilde{X}^n(t,z)\right)  {\der}z \,{\der}t.
\end{align*}
Since $(\tilde{X}^n)_{n\in\N}$ and $(\check{\rho}^n)_{n\in\N}$ converge strongly in $\Ll1([0,T]\times [0,L];\reali)$, and in view of Proposition~\ref{pro:initial}, we get by sending $n\rightarrow +\infty$
\begin{align*}
    & \int_{\reali_+}\int_0^L \left[\vphantom{\frac{\der}{{\der}t}} \varphi_t\!\left(t,X(t,z)\right)
    + v\!\left(\check{\rho} (t,z)\right)  \varphi_x\!\left(t,X(t,z)\right) \right] {\der}z \,{\der}t
    +\int_0^L \varphi\left(0,X_{\bar\rho}(z)\right)\,{\der}z
    = 0.
\end{align*}
We now apply the change of variable $x=X(t,z)$, see Lemma~\ref{lem:change_of_variable}, and obtain
\begin{align*}
    & \int_{\reali_+}\!\!\int_\reali\! \left[\vphantom{\frac{\der}{{\der}t}}  \rho\!\left(t,x\right) \varphi_t\!\left(t,x\right)
    \!+\! \rho(t,x) \, v\!\left(\check{\rho}\!\left(t,\!F(t,x)\right)\right)  \varphi_x(t,x) \right] \!{\der}x \,{\der}t
    \!+\!\!\int_\reali \!\bar\rho(x) \, \varphi(0,x) \,{\der}x
    = 0.
\end{align*}
Finally, by Proposition~\ref{pro:support} we have $\check{\rho}\left(t,F(t,x)\right)=\rho(t,x)$ a.e.~on $\spt(\rho)$, and therefore $\rho$ satisfies~\eqref{eq:weaksol}.
\end{proof}

We are now ready to complete the proof of our main result.

\begin{proof}[Proof of Theorem~\ref{thm:main}]
In view of Theorem~\ref{thm:chen}, the entropy inequality~\eqref{eq:entropy_ineq_notrace} is sufficient in order to show that $\rho$ is the unique entropy solution in the sense of Definition~\ref{def:entropy_intro}.

Let $\varphi\in C^\infty_c(\left]0,+\infty\right[\times \reali)$ with $\varphi\ge 0$ and $k\ge0$ be a constant. We shall prove that the limit $\rho$ satisfies the entropy inequality~\eqref{eq:entropy_ineq_notrace}. We consider the quantity
\begin{align*}
   & \int_{\reali_+}\!\int_{\reali}\! \left[\vphantom{\int_{\reali_+}} \modulo{\hat\rho^n(t,x) -k} \, \varphi_t(t,x) +\sgn(\hat\rho^n(t,x) -k) \left[f(\hat\rho^n(t,x)) - \!f(k)\right] \varphi_x(t,x) \right] \!{\der}x\,{\der}t\\
   &= B_0 + B_N + \sum_{i=0}^{N_n-1} I_i,
\end{align*}
with
\begin{align*}
    B_0\doteq&\int_{\reali_+}\int_{-\infty}^{x_0^n(t)}\left[\vphantom{\int_{\reali_+}} k \, \varphi_t(t,x) +f(k) \, \varphi_x(t,x)\right]\,{\der}x\,{\der}t,\\
    B_N\doteq&\int_{\reali_+}\int_{x_{N_n}^n(t)}^{+\infty}\left[\vphantom{\int_{\reali_+}} k \, \varphi_t(t,x) +f(k) \, \varphi_x(t,x)\right]\,{\der}x\,{\der}t,\\
    I_i\doteq&\int_{\reali_+}\int_{x^n_i(t)}^{x^n_{i+1}(t)} \modulo{y_i^n(t)-k} \, \varphi_t(t,x) \,{\der}x\,{\der}t \\&+\int_{\reali_+}\int_{x^n_i(t)}^{x^n_{i+1}(t)} \sgn(y^n_i(t,x)-k) \left[f(y^n_i(t,x))-f(k)\right] \varphi_x(t,x) \,{\der}x\,{\der}t.
\end{align*}
For simplicity in the notation, from now on we shall drop the $n$ index and the $(t,x)$ dependency, except in cases in which $t=0$. Moreover we define $y_N\equiv0$. We next observe by~\eqref{eq:ftl} that
\begin{align}
  & \frac{\der}{{\der}t}\left[\int_{x_i}^{x_{i+1}} \varphi \,{\der}x\right] = v(y_{i+1}) \, \varphi(t,x_{i+1}) -v(y_i) \, \varphi(t,x_i) + \int_{x_i}^{x_{i+1}}\varphi_t \,{\der}x,\label{eq:phi1}\\
    & \frac{\der}{{\der}t}\left[\int_{-\infty}^{x_{0}} \varphi \,{\der}x\right] = v(y_{0}) \, \varphi(t,x_{0}) + \int_{-\infty}^{x_{0}}\varphi_t \,{\der}x,\label{eq:phi2}\\
   & \frac{\der}{{\der}t}\left[\int_{x_{N}}^{+\infty} \varphi \,{\der}x\right] = - v_{\max} \varphi(t,x_{N}) + \int_{x_N}^{+\infty}\varphi_t \,{\der}x.\label{eq:phi3}
\end{align}
In view of~\eqref{eq:phi2} and~\eqref{eq:phi3}, the terms $B_0$ and $B_N$ can be rewritten as follows
\begin{align*}
    & B_0=\int_{\reali_+}k \left[v(k)-v(y_0)\right] \varphi(x_0)\,{\der}t,&
    & B_N=\int_{\reali_+}k \left[v_{\max}-v(k)\right] \varphi(x_N)\,{\der}t .
\end{align*}
As for the term $I_i$, we have for $i=0,\ldots,N-1$
\begin{align*}
    I_i=&\int_{\reali_+} \modulo{y_i-k}\left\{\frac{\der}{{\der}t}\left[\vphantom{\int_{x_{N-1}}^{x_N}} \int_{x_i}^{x_{i+1}} \varphi \,{\der}x \right]-v(y_{i+1}) \, \varphi(x_{i+1}) +v(y_i) \, \varphi(x_i)\right\} {\der}t\\
    &+\int_{\reali_+} \sgn(y_i -k) \left[f(y_i) - f(k)\right] \left[\varphi(x_{i+1})-\varphi(x_i)\right] {\der}t.
\end{align*}
By~\eqref{eq:dyi}, we compute the term
\begin{align*}
    &\int_{\reali_+}\modulo{y_i-k} \frac{\der}{{\der}t} \left[\int_{x_i}^{x_{i+1}}\varphi \,{\der}x \right] {\der}t
    = -\int_{\reali_+} \left[\int_{x_i}^{x_{i+1}}\varphi \,{\der}x \right] \frac{\der}{{\der}t}\modulo{y_i-k}\,{\der}t \\
    &= -\int_{\reali_+} \sgn(y_i-k) \left[-\frac{y_i^2}{\ell}[v(y_{i+1})-v(y_i)]\right] \left[\int_{x_i}^{x_{i+1}}\varphi \,{\der}x \right] {\der}t\\
    &= \int_{\reali_+}\sgn(y_i-k) \, y_i \left[v(y_{i+1})-v(y_i)\right] \left[\fint_{x_i}^{x_{i+1}}\varphi \,{\der}x \right] {\der}t.\\
\end{align*}
Hence, we have
\begin{align*}
    & \sum_{i=0}^{N-1} I_{i-1} = \sum_{i=1}^{N} \int_{\reali_+} \!\!\!J_i\,{\der}t
    + \sum_{i=1}^{N} \int_{\reali_+} \!\!\!K_i \, \varphi(x_i) \,{\der}t
    + \!\int_{\reali_+} \!\!\!L \, \varphi(t,x_0) \,{\der}t
    -\!\int_{\reali_+} \!\!\!M \, \varphi(t,x_N) \,{\der}t,
\end{align*}
with
\begin{align*}
    J_i&\doteq \sgn(y_{i-1} -k) \, y_{i-1} \left[v(y_i)-v(y_{i-1})\right] \left[\fint_{x_{i-1}}^{x_i}\varphi\,{\der}x - \varphi(x_i)\right] ,\\
    K_i&\doteq \sgn(y_{i-1}-k) \, y_{i-1} \left[v(y_i)-v(y_{i-1})\right] + \modulo{y_i-k} \, v(y_i)\\
    &-\sgn(y_i-k)[f(y_i)-f(k)] \!-\! \modulo{y_{i-1}-k} \, v(y_i) + \sgn(y_{i-1}-k) \left[f(y_{i-1})-f(k)\right],\\
    L&\doteq \modulo{y_0-k} \, v(y_0)-\sgn(y_0-k) \left[f(y_0)-f(k)\right],\\
    M&\doteq k \, v_{\max}-f(k).
\end{align*}
We observe that
\begin{equation*}
    B_N -\int_{\reali_+} M \varphi(t,x_N) \,{\der}t = 0.
\end{equation*}
We now compute $L$. If $k<y_0$, we have
\begin{align*}
   & L = k \left[v(k)-v(y_0)\right] \ge 0,
\end{align*}
as $v$ is non increasing. Therefore, for $k<y_0$
\begin{align*}
    & B_0+ \int_{\reali_+} L \, \varphi(t,x_0)\,{\der}x = 2\int_{\reali_+} k \left[v(k)-v(y_0)\right] \varphi(x_0) \,{\der}t \ge 0 .
\end{align*}
Similarly, for $k\ge y_0$ we have
\begin{align*}
   & L = k \left[v(y_0)-v(k)\right] \ge 0,
\end{align*}
which gives
\begin{align*}
    & B_0+ \int_{\reali_+} L \, \varphi(t,x_0)\,{\der}x = 0.
\end{align*}
We now compute the term $K_i$ for $i=1,\ldots,N$. After some easy manipulations, we get
\begin{align*}
  K_i = k \left[v(k)-v(y_i)\right] \left\{\sgn(y_i-k)-\sgn(y_{i-1}-k)\right\}.
\end{align*}
We consider all the possible cases for $k$. If either $k<\min\{y_i,y_{i-1}\}$, or $k>\max\{y_i,y_{i-1}\}$, then we easily get $K_i=0$. If $y_i\le k \le y_{i-1}$, then $K_i=2k[v(y_{i})-v(k)] \ge 0$ as $v$ is non increasing. Finally, if $y_{i-1}\le k \le y_{i}$, then $K_i = 2k[v(k)-v(y_i)]\ge 0$. In all cases, we get $K_i\ge 0$ for all $i=1,\ldots,N$. Putting all the terms together, we get
\begin{align}
    & \int_{\reali_+}\int_{\reali} \left[\vphantom{\int_{\reali_+}} \modulo{\hat\rho -k} \, \varphi_t +\sgn(\hat\rho -k) \left[f(\hat\rho) - f(k)\right] \varphi_x \right] {\der}x\,{\der}t
    \ge  \sum_{i=1}^N \int_{\reali_+} J_i \,{\der}t  .\label{eq:entropy_almost_final}
\end{align}
We now estimate the terms $J_i$. For some $\delta>0$, assuming that the support of $\varphi$ is contained in the strip $t\in [\delta,T]$, we have by Proposition~\ref{pro:compactness}
\begin{align*}
    \modulo{\sum_{i=1}^N \int_{\reali_+} \!\!\!\!J_i \,{\der}t}
    &= \modulo{\sum_{i=1}^N \int_{\reali_+} \!\!\!\!\sgn(y_{i-1} -k) \, y_{i-1} \left[v(y_i)-v(y_{i-1})\right] \left[\fint_{x_{i-1}}^{x_i}\!\!\!\!\varphi\,{\der}x - \varphi(x_i)\right] {\der}t}\\
    &\le \int_{\reali_+} \sum_{i=1}^N \left[ \frac{y_{i-1}^2}{\ell} \,  \modulo{v(y_i)-v(y_{i-1})} \int_{x_{i-1}}^{x_i}\modulo{\varphi(x)-\varphi(x_i)}\,{\der}x\,\right] {\der}t\\
    &\le \lip(\varphi) \int_\delta^{T} \sup_{i=1,\ldots,N}\left[\frac{y_{i-1}^2 \, \left(x_i-x_{i-1}\right)^2}{\ell}\right]\sum_{i=1}^N \modulo{v(y_i)-v(y_{i-1})}\,{\der}t
    \\
    &\le \ell \, \lip(\varphi) \, T \, \sup_{t\ge \delta}\tv\left[v(\hat{\rho}^n(t))\right]
    \le \ell \, \lip(\varphi) \, T \, C_\delta.
\end{align*}
As a consequence
\[\lim_{n\rightarrow +\infty}  \sum_{i=1}^N \int_{\reali_+} J_i \,{\der}t = 0\]
and letting $n$ go to infinity in~\eqref{eq:entropy_almost_final} we obtain the entropy inequality~\eqref{eq:entropy_ineq_notrace}.
\end{proof}

\subsection{Concluding remarks}

We conclude this paper with the some technical remarks which help motivating our choices in the strategy of the proof at several stages in the paper.
\begin{itemize}
  \item In the case of $v$ such that $v'\le -c<0$, then the Oleinik-type estimate~\eqref{eq:oleinik_discrete_2} gives a one sided estimate for $\hat\rho^n_x$ in the sense of distributions. Such an estimate can be then passed to the limit very easily, and one obtains an analogous estimate for the limit. In this way, one can check that the limit $\rho$ is an entropy solution in much easier way than the above proof. In the general case of $v'$ possibly degenerating, such a strategy fails. Indeed, surprisingly enough the Oleinik estimate one gets in the limit from~\eqref{eq:oleinik_discrete_2} is \emph{not} equivalent (in general) to the estimate~\eqref{eq:ole_intro_hoff}. For this reason, we preferred getting the entropy condition in the Kru{\v{z}}kov sense rather than the one sided Lipschitz condition. This strategy allows in particular to get the entropy condition in the limit also in the case of $v$ not satisfying (V3) and $\bar\rho$ satisfying (InBV).
  \item In the case of linear velocity $v$, e.g.~$v(\rho)=v_{\max}(1-\rho)$, the convergence to a weak solution~\eqref{eq:weaksol} can be obtained without the need of the $\BV$ estimates, as the velocity term in~\eqref{eq:approx_PDE} is linear. This is somehow intrinsic in using a Lagrangian description.
  \item In order to get continuity in time for the sequence $\hat{\rho}^n$, the most natural try would be getting $\Ll1$--continuity. Encouraged by the $\Ll1$ time equi-continuity of $\check{\rho}^n$, we have attempted at proving such a property in many ways without success. This is the reason why use the generalized Aubin-Lions lemma, which allows to take advantage of the Wasserstein equi-continuity of $\hat{\rho}^n$, and still get the same $\Ll1$--compactness in the end. The only drawback of this strategy is that we can't get any $\Ll1$ time continuity for the limit.
  \item As pointed out in the introduction, the proposed Lagrangian approach has the advantage of providing a piecewise constant approximation with a non increasing number of jumps. The price to pay for such a simplification is that we lose the classical shock structure at a microscopic level. Indeed, as pointed out in~\cite{ColomboRossi,Rossi}, the explicit solution to the FTL system even for simple Riemann--type initial conditions is not immediate. On the other hand, this aspect gives an added value to our result, as we show that shocks and rarefaction waves are still achieved in the macroscopic limit, despite not being easily detectable at the microscopic level.
  \item We finally recall that a symmetric construction can be set in the case of monotone increasing velocities. In that case, the suitable particle system should be recursively defined `from the left towards the right', and therefore each particles adjusts its velocity according to the distance of the particle at its left hand side.
\end{itemize}

\appendix

\section{Heuristic derivation of the FTL~model from the LWR~model}\label{sec:formalderivation1}

In this appendix we formally provide our derivation from the scalar conservation law~\eqref{eq:LWRmodel} of a discrete approximating model of the form~\eqref{eq:FTL_intro}.

Let $\rho$ be an entropy solution of~\eqref{eq:LWRmodel} in the sense of Definition~\ref{def:entropy_intro}. We assume for simplicity that $\rho$ is compactly supported. The total mass of $\rho$ in in $\left]-\infty, x\right]$ at time $t \ge0$ is
\begin{align}\label{eq:Fc}
    F(t,x) \doteq \int_{-\infty}^x \rho(t,y) \,{\der}y.
\end{align}
Clearly, $F$ takes values in $[0,L]$, where $L \doteq \norma{\bar\rho}_{\Ll1(\reali;[0,1])}$, and for any fixed $t\ge0$ the map $x \mapsto F(t,x)$ is continuous and non-decreasing, $F(t,-\infty)=0$ and $F(t,+\infty) = L$. The result in the next proposition shows that~\eqref{eq:LWR_intro1} is equivalent to requiring that the weak partial derivatives of $F$ with respect to time and space commute in the sense of distributions.
\begin{proposition}[\cite{courant1976supersonic}]\label{prop:F}
    The partial derivatives of $F$ satisfy in the sense of distributions
    \begin{align}\label{eq:prop_app}
        &F_x = \rho,&
        F_t = -f\left(\rho\right) .
    \end{align}
\end{proposition}

\begin{proof}

The first equality in~\eqref{eq:prop_app} is obvious. For any test function $\psi \in \Cc\infty(\left]0,+\infty\right[\times\reali;\reali)$ we have that by~\eqref{eq:weaksol}
    \begin{align*}
        &\int_{\reali} \int_{\reali_+} F(t,x) \,\partial_t\psi_x(t,x) \,{\der}t \,{\der}x
        =\int_{\reali} \int_{\reali_+} F(t,x) \,\partial_x\psi_t(t,x) \,{\der}t \,{\der}x
        \\
        &= -\int_{\reali} \int_{\reali_+} \rho(t,x) \,\psi_t(t,x) \,{\der}t \,{\der}x
        = \int_{\reali} \int_{\reali_+} f\left(\rho(t,x)\right) \psi_x(t,x) \,{\der}t \,{\der}x.
    \end{align*}
    This shows that for any $t\ge0$, the map $x \mapsto \left[ F_t(t,x) + f\left(\rho(t,x)\right)\right]$ is constant (as a distribution). Therefore there exists $c \in \Llloc1(\left[0,+\infty\right[;\reali)$ such that
    \begin{align*}
        & \int_{\reali} \int_{\reali_+} \left[\vphantom{\int_{\reali_+}} F(t,x) \, \varphi_t(t,x) - f\left(\rho(t,x)\right)  \varphi(t,x) + c(t) \, \varphi(t,x)\right] {\der}t \,{\der}x = 0.
    \end{align*}
    Choose now, for any integer $k\in \N$,
    \begin{align*}
        \varphi(t,x)=\eta(t) \, \psi(x-k),
    \end{align*}
    where $\eta\in \Cc\infty\left(\left]0,+\infty\right[;\reali\right)$ and $\psi\in \Cc\infty(\reali;\left[0,+\infty\right[)$ such that $\norma{\psi}_{\Ll1(\reali;\reali)} = 1$. We get
    \begin{align*}
       0 & = \int_{\reali} \int_{\reali_+} \left[\vphantom{\int_{\reali_+}}  F(t,x) \, \dot\eta(t) - f\left(\rho(t,x)\right)  \eta(t) + c(t) \, \eta(t)\right]\psi(x-k) \,{\der}t \,{\der}x \\
        & \ = \int_{\reali} \int_{\reali_+} \left[\vphantom{\int_{\reali_+}}  F(t,x+k) \, \dot\eta(t) - f\left(\rho\left(t,x+k\right)\right)  \eta(t) + c(t) \, \eta(t)\right]\psi(x)\,{\der}t \,{\der}x.
    \end{align*}
    By Lebesgue dominated convergence theorem, we can send $k$ to $+\infty$ and get
    \begin{align*}
       & 0 = \int_{\reali} \int_{\reali_+} \left[\vphantom{\int_{\reali_+}}  L \, \dot\eta(t) +c(t) \, \eta(t)\right]\psi(x)\,{\der}t \,{\der}x
       = \int_{\reali_+} c(t) \, \eta(t) \,{\der}t,
    \end{align*}
    and the above expression on the right hand side can be easily made nonzero by suitably choosing $\eta$, unless $c(t)=0$ for a.e.~$t\ge0$, which proves the assertion.
\end{proof}

For any $t\ge0$ the map $x \mapsto F(t,x)$ is strictly increasing on the intervals where the density $x \mapsto \rho(t,x)$ is not zero and otherwise it is constant. Therefore we can introduce $X \doteq \mathcal{X}[F]$, the pseudo-inverse of $F$. Now, assume for simplicity that $\rho\left(t, x\right) > 0$ for all $(t,x) \in \spt(\rho) = \left\{(t,x) \in \reali_+ \times\reali \,\colon\, a(t) \le x \le b(t)\right\}$. Then, for any $t\ge0$ by Proposition~\ref{prop:F} we have that $x \mapsto F(t,x)$ is strictly increasing on $\spt\left(\rho(t)\right)$. This implies that $\mathcal{X}[F]$ is the inverse of $F$ on the support of $\rho$, namely  $F\left(t, X(t,z)\right) = z$ on $(t,z)\in \reali_+\times\left[0,L\right]$, and, assuming that all the derivatives below are well defined, we have that
\begin{align*}
    &F_x\left(t, X(t,z)\right) = \rho\left(t, X(t,z)\right) >0 &
    \hbox{for a.e.~}(t,z) \in \reali_+\times\left[0,L\right] .
\end{align*}
Therefore,
\begin{align*}
    & 1 = \frac{\der}{{\der}z} F\left(t, X(t,z)\right) = F_x\left(t, X(t,z)\right)  X_z(t,z),\\
    & 0 = \frac{\der}{{\der}t} F\left(t, X(t,z)\right) = F_t\left(t, X(t,z)\right) + F_x\left(t, X(t,z)\right)  X_t(t,z),
\end{align*}
which yields, once again by Proposition~\ref{prop:F}, that $X(t,z)$ is indeed a solution of the PDE
\begin{align}\label{eq:LWR_Lagrangian}
  &X_t(t,z)= v\left(\frac{1}{X_z(t,z)}\right).
\end{align}
The initial condition $X(0,z)$ is determined by
\begin{equation*}
   \int_{-\infty}^{X(0,z)} \bar\rho(y) \,{\der}y = z.
\end{equation*}
The computation above is only rigorous on the sets in which $\rho\left(t, x\right) > 0$.

The last step needed in order to (formally) recognize the discrete model~\eqref{eq:FTL_intro} in~\eqref{eq:LWR_Lagrangian} is by replacing the $z$--derivative of $X$ in~\eqref{eq:LWR_Lagrangian} by the (forward) finite differences
\begin{equation}\label{eq:finite_diff0}
    X_z\approx \frac{X(t,z+\ell)-X(t,z)}{\ell},
\end{equation}
which gives
\begin{align*}
    & X_t(t,z) \approx v\left(\frac{\ell}{X(t,z+\ell)-X(t,z)}\right).
\end{align*}
Then the desired model~\eqref{eq:FTL_intro} is obtained by assuming that $X(t)$ is piecewise constant on intervals of measure $\ell$, with $X(t,j \, \ell)=x_j(t)$, $j=1,\ldots,N-1$. For any fixed $z \in \{i\,\ell \,\colon\, i=0,1, \ldots, N\}$, the map $t\mapsto X(t,z)$ can be ideally interpreted as the path described by the `infinitesimal particle' labelled with $z\in [0,L]$. Therefore, \eqref{eq:LWR_Lagrangian} can be interpreted as the expression in the Lagrangian coordinates $(t,z)$ of the Cauchy problem~\eqref{eq:LWRmodel}.

\section*{acknowledgements}
Both the authors acknowledge useful discussions with J.A.~Carrillo, R.M.~Colombo, P.~Degond, P.~Marcati, and D.~Matthes. M.~Di Francesco is supported by the Marie Curie CIG (Career Integration Grant) DifNonLoc - Diffusive Partial Differential Equations with Nonlocal Interaction in Biology and Social Sciences and by the Ministerio de Ciencia e Innovaci\'{o}n, grant MTM2011-27739-C04-02. M.D.~Rosini is supported by ICM. Projekt zosta{\l} sfinansowany ze \'srodk\'ow Narodowego Centrum Nauki przyznanych na podstawie decyzji nr: DEC-2011/01/B/ST1/03965.

%% BibTeX users please use one of
%%\bibliographystyle{spbasic}      % basic style, author-year citations
%%\bibliographystyle{spphys}       % APS-like style for physics
%\bibliography{bibliografia}   % name your BibTeX data base

%% Non-BibTeX users please use
%\begin{thebibliography}{}
%%
%% and use \bibitem to create references. Consult the Instructions
%% for authors for reference list style.
%\end{thebibliography}

\end{document}